\documentclass[11pt]{amsart}

\usepackage[a4paper,hmargin=3.5cm,vmargin=3.9cm]{geometry}
\usepackage{amsfonts,amssymb,amscd,amstext}
\usepackage{graphicx}
\usepackage[dvips]{epsfig}
\usepackage{quoting}
\usepackage{hyperref}

\usepackage{fancyhdr}
\input xy
\xyoption{all}

\usepackage{enumerate}
\usepackage{titlesec}
\usepackage{mathrsfs}

\titleformat{\section}
{\filcenter\bfseries\large} {\thesection{.}}{0.2cm}{}
\titleformat{\subsection}[runin]
{\bfseries} {\thesubsection{.}}{0.15cm}{}[.]
\titleformat{\subsubsection}[runin]
{\em}{\thesubsubsection{.}}{0.15cm}{}[.]

\usepackage[up,bf]{caption}

\newtheorem{theorem}{Theorem}[section]
\newtheorem{proposition}[theorem]{Proposition}
\newtheorem{lemma}[theorem]{Lemma}
\newtheorem{claim}[theorem]{Claim}
\newtheorem{corollary}[theorem]{Corollary}

\theoremstyle{definition}

\newtheorem{remark}[theorem]{Remark}

\numberwithin{equation}{section}
\numberwithin{figure}{section}

\newcommand\Cscr{\mathscr{C}}

\newcommand\Oscr{\mathscr{O}}

\newcommand\C{\mathbb{C}}

\newcommand\N{\mathbb{N}}

\newcommand\R{\mathbb{R}}

\newcommand\Z{\mathbb{Z}}

\newcommand\igot{\mathfrak{i}}

\renewcommand\igot{\mathfrak{i}}

\renewcommand\imath{\igot}

\newcommand\dist{\mathrm{dist}}

\newcommand\length{\mathrm{length}}

\newcommand\Flux{\mathrm{Flux}}

\newcommand\CMI{\mathrm{CMI}}
\newcommand\NC{\mathrm{NC}}

\def\dist{\mathrm{dist}}

\def\length{\mathrm{length}}
\def\Flux{\mathrm{Flux}}

\newcommand\CMIf{\mathrm{CMI_{full}}}
\newcommand\CMIfc{\mathrm{CMI_{full}^c}}
\newcommand\CMInf{\mathrm{CMI_{nf}}}
\newcommand\CMInfc{\mathrm{CMI_{nf}^c}}
\newcommand\NCf{\mathrm{NC_{full}}}

\newcommand\RNCf{\Re\mathrm{NC_{full}}}
\newcommand\RNCfc{\Re\mathrm{NC_{full}^c}}
\newcommand\NCnf{\mathrm{NC_{nf}}}
\newcommand\NCnfc{\mathrm{NC_{nf}^c}}
\newcommand\RNC{\Re\mathrm{NC}}
\newcommand\RNCnf{\Re\mathrm{NC_{nf}}}
\newcommand\RNCnfc{\Re\mathrm{NC_{nf}^c}}
\newcommand\Of{\mathscr O_{\mathrm{full}}}
\newcommand\Onf{\mathscr O_{\mathrm{nf}}}
\newcommand\boldA{\mathbf A}

\usepackage{color}

\begin{document}


\fancyhead[LO]{A strong parametric h-principle for complete minimal surfaces}
\fancyhead[RE]{A.\ Alarc\'on and F.\ L\'arusson}
\fancyhead[RO,LE]{\thepage}

\thispagestyle{empty}



\begin{center}
{\bf \LARGE A strong parametric h-principle
\\ \vspace{1mm}  
for complete minimal surfaces} 

\medskip

%
%
{\large\bf Antonio Alarc\'on \; and \; Finnur L\'arusson}
\end{center}


%
%

\medskip

\begin{quoting}[leftmargin={5mm}]
{\small
\noindent {\bf Abstract}\hspace*{0.1cm}
We prove a parametric h-principle for complete nonflat conformal minimal immersions of an open Riemann surface $M$ into $\R^n$, $n\geq 3$.  It follows that the inclusion of the space of such immersions into the space of all nonflat conformal minimal immersions is a weak homotopy equivalence.  When $M$ is of finite topological type, the inclusion is a genuine homotopy equivalence.  By a parametric h-principle due to Forstneri\v c and L\'arusson, the space of complete nonflat conformal minimal immersions therefore has the same homotopy type as the space of continuous maps from $M$ to the punctured null quadric.  Analogous results hold for holomorphic null curves $M\to\C^n$ and for full immersions in place of nonflat ones.


\noindent{\bf Keywords}\hspace*{0.1cm} 
Riemann surface, minimal surface, complete minimal surface, null curve, complete null curve, h-principle, Oka manifold, absolute neighbourhood retract


\noindent{\bf Mathematics Subject Classification (2020)}\hspace*{0.1cm} 
Primary 53A10. Secondary 30F99, 32E30, 32H02, 32Q56, 54C55, 55M15

}

\end{quoting}



\section{Introduction and main results}
\label{sec:intro}

\noindent
Over the past ten years or so, powerful complex-analytic methods from Oka theory have been introduced and applied in the classical theory of minimal surfaces in Euclidean spaces.  For an overview of this development, see the survey \cite{AlarconForstneric2019JAMS}.  For a detailed exposition, see the monograph \cite{AlarconForstnericLopez2021}.  Complete surfaces are of central importance in Riemannian geometry and in particular in the theory of minimal surfaces. Some of the fundamental results on complete minimal surfaces that have been proved using Oka theory are the following.  Here, $M$ denotes an open Riemann surface, always assumed connected, and $n\geq 3$.
\begin{enumerate}[A]
\item[$\bullet$]  The space $\CMInfc(M, \R^n)$ of complete nonflat conformal minimal immersions $M\to\R^n$ is dense (with respect to the compact-open topology) in the space $\CMInf(M,\R^n)$ of all nonflat conformal minimal immersions (\cite[Theorem 7.1]{AlarconForstnericLopez2016MZ}; the case of $n=3$ follows from \cite[Theorem 5.6]{AlarconLopez2012JDG}, which slightly pre\-dates the introduction of Oka theory in minimal surface theory).  In more recent work, the density theorem has been strengthened to Mergelyan and Carleman approximation theorems including Weierstrass interpolation and other additional features (see \cite{AlarconCastro-Infantes2019APDE}, \cite[Section 3.9]{AlarconForstnericLopez2021}, and \cite{Castro-InfantesChenoweth2020}).
\medskip
\item[$\bullet$]  Every nonflat conformal minimal immersion $M\to\R^n$ can be deformed, through such immersions, to a complete one (the case of $n=3$ is part of \cite[Theorem 1.2]{AlarconForstneric2018Crelle}; the proof there is easily adapted to the general case).  In other words, the inclusion $\CMInfc(M,\R^n)\hookrightarrow \CMInf(M,\R^n)$ induces a surjection of path components.  As far as we know, further homotopy-theoretic properties of this inclusion have not been studied in any previous work.
\end{enumerate}

Recall that a conformal immersion $u\colon M\to\R^n$ is minimal if and only if it is a harmonic map. Such an immersion is said to be flat if it maps $M$ into an affine 2-plane in $\R^n$.  Equivalently, the holomorphic map $\partial u/\theta$ from $M$ into the punctured null quadric $\boldA_*=\{z\in\C^n:z_1^2+\cdots+z_n^2=0, z\neq 0\}$ is flat, that is, maps $M$ into an affine complex line in $\C^n$ \cite[Lemma 12.2]{Osserman1986}.  Here, $\theta$ is a nowhere-vanishing holomorphic 1-form on $M$, chosen once and for all, and we denote by $\partial u$ the $(1,0)$-differential of $u$.  Nonflatness is a very mild and natural nondegeneracy condition.  Its key significance in Oka-theoretic proofs is that it allows $\partial u/\theta$ to be realised as the core of a period dominating spray of holomorphic maps into  the Oka manifold $\boldA_*$ (such sprays first appeared in \cite[Lemma 5.1]{AlarconForstneric2014IM}).  Recall  also that the flux $\Flux(u)$ of a conformal minimal immersion $u\colon M\to\R^n$ is the cohomology class of its conjugate differential $d^cu=\imath(\bar\partial u-\partial u)$ in $H^1(M,\R^n)$.  The flux is naturally identified with the group homomorphism $\Flux(u)\colon H_1(M,\Z)\to\R^n$ given by
\[
	\Flux(u)([C]):=\int_C d^c u=-2\imath\int_C \partial u,\quad [C]\in H_1(M,\Z).
\]
We view the cohomology group $H^1(M,\C^n)$ as the de Rham group of $n$-tuples of holomorphic $1$-forms on $M$ modulo exact forms, with the quotient topology induced from the compact-open topology.  The subgroup $H^1(M,\R^n)$ carries the subspace topology.

Our first theorem is a strong parametric h-principle for complete minimal surfaces that subsumes as very particular consequences the density and deformation results described above.
%
%
\begin{theorem}  \label{th:main}  
Let $M$ be an open Riemann surface, $P$ be a compact metric space, and $u\colon M\times P\to\R^n$, $n\ge 3$, be a continuous map such that $u_p:=u(\cdot,p)\colon M\to\R^n$ is a nonflat conformal minimal immersion for all $p\in P$. 

If $K\subset M$ is compact and $Q\subset P$ is closed, then for any $\epsilon>0$ there is a homotopy $u^t\colon M\times P\to\R^n$, $t\in [0,1]$, satisfying the following conditions.
\begin{enumerate}[\rm (i)]
\item  The map $u_p^t:=u^t(\cdot,p)\colon M\to\R^n$ is a nonflat conformal minimal immersion for all $(p,t)\in P\times[0,1]$.
\smallskip
\item  $u_p^t=u_p$ for all $(p,t)\in (P\times\{0\})\cup (Q\times[0,1])$.
\smallskip
\item  $|u_p^t(x)-u_p(x)|<\epsilon$ for all $x\in K$ and $(p,t)\in P\times [0,1]$.
\smallskip
\item  $u_p^t$ is complete for all $(p,t)\in (P\setminus Q)\times (0,1]$.
\end{enumerate}
Furthermore, given a homotopy $F^t\colon P\to H^1(M,\R^n)$, $t\in [0,1]$, such that $F^t(p)=\Flux(u_p)$ for all $(p,t)\in (P\times\{0\})\cup (Q\times[0,1])$ and $F^t(p)\vert_K=\Flux(u_p)\vert_K$ for all $(p,t)\in P\times [0,1]$, we can choose $u^t$ such that
\begin{enumerate}[\rm (i)]
\item[\rm (v)]  $\Flux(u_p^t)= F^t(p)$ for all $(p,t)\in P\times [0,1]$.
\end{enumerate}
\end{theorem}
Note that the sets $(P\times\{0\})\cup (Q\times[0,1])$ in (ii) and $(P\setminus Q)\times (0,1]$ in (iv) partition the parameter space $P\times [0,1]$.

A parametric h-principle for complete nonflat conformal minimal immersions, formulated as parametric h-principles usually are, would provide condition {\rm (iv)} only for the immersions $u_p^1$, $p\in P$, under the additional assumption that the given immersions $u_p$ are complete for all $p\in Q$.  Theorem \ref{th:main} is much stronger than this, which is the reason for the term {\em strong} in the title of this paper.  Even the basic h-principle for complete minimal surfaces that we obtain from Theorem \ref{th:main} by taking $P$ to be a singleton and $Q$ to be empty is a considerable improvement on the strongest previously known result in this direction, which is \cite[Theorem 1.2]{AlarconForstneric2018Crelle}.  It was proved by means of the Oka principle for sections of ramified holomorphic maps with Oka fibres (see \cite{Forstneric2003FM} or \cite[Section 6.13]{Forstneric2017E}), a tool that is not available in our general parametric setting.  

The earliest examples of homotopy principles (h-principles for short) that the authors are aware of are, in the real setting, the Whitney-Graustein theorem of 1937, stating that smooth immersions of the circle in the plane are classified up to isotopy by the winding numbers of their tangent maps, and, in the complex setting, Oka's theorem of 1939, stating, in modern terms, that a holomorphic line bundle on a Stein manifold is trivial if it is topologically trivial.  The former result was the beginning of a vast program of research within differential topology; modern Oka theory has its roots in the latter.  Around 1970, Gromov formalised the concept of an h-principle for a partial differential relation as saying that every formal solution of the relation can be deformed to a genuine solution (see \cite{EliashbergMishachev2002, Gromov1973, Gromov1986}).  The obstruction to the existence of a formal solution is usually purely topological, and if it vanishes, then a genuine solution exists.  A parametric h-principle deals with families of solutions depending on a parameter in a space that is almost always compact.  It means that the inclusion of the space of genuine solutions into the space of formal solutions is a weak homotopy equivalence.  This kind of principle is most clearly reflected in our Corollary \ref{cor:homotopy-type-of-complete} below: a formal conformal minimal immersion $M\to\R^n$ (complete or not) can, using the trivialisation of the cotangent bundle of $M$ given by a form $\theta$ as above, be viewed as a continuous map $M\to\boldA_*$.  It is noteworthy that the applications of Oka theory in the theory of minimal surfaces (such as here, in \cite{ForstnericLarusson2019CAG}, and going back to \cite{AlarconForstneric2014IM}) also involve an h-principle from real analysis, namely Gromov's h-principle for ample partial differential relations, proved using his method of convex integration.  The prototypical example of such an h-principle is the Whitney-Graustein theorem. 

Theorem \ref{th:main} is proved in Sections \ref{sec:lemma}, \ref{sec:flux}, and \ref{sec:proof}. In Section \ref{sec:lemma}, which is the core of the paper, we obtain a parametric completeness lemma to the effect that, given compact Hausdorff spaces $Q\subset P$ and a homotopy of nonflat conformal minimal immersions $u_p^t\colon L\to\R^n$, $(p,t)\in P\times [0,1]$, on a compact domain $L$ in an open Riemann surface, one can deform the homotopy near the boundary of $L$ in order to arbitrarily increase the boundary distance from a fixed interior point of all the immersions $u_p^t$ with $(p,t)$ outside a neighbourhood of $(P\times\{0\})\cup (Q\times[0,1])$ while keeping fixed those with $(p,t)$ in that set; see Lemma \ref{lem:distance}. The proof relies on a finite recursive application of a sort of parametric L\'opez-Ros deformation for minimal surfaces in $\R^n$ which we develop in Lemma \ref{lem:pair}; we refer to the beginning of Section \ref{sec:lemma} for a brief explanation.  In Section \ref{sec:flux} we extend the arguments in \cite{ForstnericLarusson2019CAG} in order to control the flux of all the immersions in the homotopy; see Lemma \ref{lem:flux}. Finally, we prove Theorem \ref{th:main} in Section \ref{sec:proof} by a standard inductive application of the results in Sections \ref{sec:lemma} and \ref{sec:flux}.

Part (a) of the following corollary to Theorem \ref{th:main} is immediate.  Part (b) is proved in Section \ref{sec:finite} using a method first developed in \cite{Larusson2015PAMS}.  The mapping spaces considered here are too large to have a CW structure, but when the open Riemann surface $M$ has finite topological type, an h-principle can be used to show that they are absolute neighbourhood retracts and therefore have the homotopy type of a CW complex.  The Whitehead lemma then implies that a weak homotopy equivalence between them is a genuine homotopy equivalence.

\begin{corollary}  \label{cor:weak-eq-for-minimal}
Let $M$ be an open Riemann surface and $n\geq 3$.  

{\rm (a)}  The inclusion $\CMInfc(M,\R^n) \hookrightarrow \CMInf(M,\R^n)$ is a weak homotopy equivalence.

{\rm (b)}  If $M$ is of finite topological type, then the inclusion is a homotopy equivalence.
\end{corollary}

Part (a) means that the inclusion induces a bijection of path components $\pi_0(\CMInfc(M,\R^n)) \to \pi_0(\CMInf(M,\R^n))$ and an isomorphism of homotopy groups
\[
	\pi_k(\CMInfc(M,\R^n),u) \longrightarrow \pi_k(\CMInf(M,\R^n),u)
\]
for every $k\geq 1$ and every base point $u\in \CMInfc(M,\R^n)$.
 By (b), when $M$ is of finite topological type, there is a homotopy inverse $\eta:\CMInf(M, \R^n) \to \CMInfc(M, \R^n)$ to the inclusion.  This means that there is a way to associate to every immersion $u$ a complete immersion $\eta(u)$ that is homotopic to $u$.  Moreover, if $u$ is complete to begin with, then there is such a homotopy through complete immersions.  The main point is that $\eta(u)$ and the homotopies depend continuously on $u$.   
 
  \begin{remark}
Theorem \ref{th:main} implies the following stronger version of Corollary \ref{cor:weak-eq-for-minimal}(a).  If $X$ is a subspace of $\CMInf(M,\R^n)$ containing $\CMInfc(M,\R^n)$, then the inclusions $\CMInfc(M,\R^n) \hookrightarrow X\hookrightarrow \CMInf(M,\R^n)$ are weak homotopy equivalences.
 \end{remark}

By the next corollary, which is a direct consequence of Theorem \ref{th:main},  $\CMInfc(M, \R^n)$ is dense in $\CMInf(M, \R^n)$ in a strong sense.

\begin{corollary}\label{co:extension}
If $M$ is an open Riemann surface and $Q\subset P$ are compact metric spaces such that $Q$ is a retract of $P$, then every continuous map $Q\to\CMInf(M,\R^n)$, $n\ge 3$, extends to a continuous map $P\to\CMInf(M,\R^n)$ that takes $P\setminus Q$ into $\CMInfc(M,\R^n)$. 
\end{corollary}

We now proceed to discuss the implications of condition {\rm (v)} in Theorem \ref{th:main}.  It may be seen from the results in \cite{AlarconLarusson2017IJM} that the flux map $\CMInf(M, \R^n)\to H^1(M, \R^n)$ sending an immersion $u$ to the cohomology class of $d^cu$ is a Serre fibration, that is, satisfies the homotopy lifting property with respect to all CW complexes. Next we use Theorem \ref{th:main} to prove that this also holds for the subspace of complete immersions.  If we ignore completeness, the same simple argument gives a new proof that the flux map on $\CMInf(M, \R^n)$ is a Serre fibration.

\begin{theorem}   \label{th:Serre}
If $M$ is an open Riemann surface and $n\ge 3$, then the flux map $\Flux\colon \CMInfc(M,\R^n)\to H^1(M,\R^n)$ is a Serre fibration.
\end{theorem}

\begin{proof}
Let $j:Q\hookrightarrow P$ be the inclusion in a CW complex of a subcomplex, such that $j$ is a homotopy equivalence, or simply let $j$ be the inclusion of $[0,1]^k\times\{0\}$ in $[0,1]^{k+1}$ for some $k\geq 0$, and consider a commuting square of continuous maps as follows.
\[ 
	\xymatrix{
	Q \ar@{^{(}->}[d] \ar[r]^{\!\!\!\!\!\!\!\!\!\!\!\!\!\!\!\!\!\!\!\! u} & \CMInfc(M,\R^n) \ar[d]^\Flux
		\\ 
	P \ar[r]^{\!\!\!\!\!\!\!\!\!\!\!\!\!\! f} & H^1(M,\R^n)
	} 
\]
Let $\rho:P\to Q$ be a retraction.  The map $u\circ\rho$ extends $u$.  The maps $\Flux\circ u\circ\rho$ and $f$ agree on $Q$ and are therefore homotopic relative to $Q$.  By Theorem \ref{th:main} with $K=\varnothing$, $u\circ\rho$ can be deformed, relative to $Q$, to a map $P\to\CMInfc(M,\R^n)$ with flux $f$.  Such a map is the desired lifting in the square.
\end{proof}

Let $F\in H^1(M,\R^n)$.  Theorem \ref{th:Serre} implies that the weak homotopy type of the space of complete nonflat conformal minimal immersions $M\to\R^n$ with flux $F$ is the same for all $F$.  Without completeness, this was proved in \cite{AlarconLarusson2017IJM}.  In what follows, we will focus on immersions with $F=0$, although our results hold for arbitrary $F$.

Recall that a harmonic map $u:M\to\R^n$ has a harmonic conjugate if and only if the cohomology class of $d^cu$ vanishes.  If $u\in\CMI(M,\R^n)$ has a harmonic conjugate $v$, then the holomorphic immersion $\Phi=u+iv:M\to\C^n$ is a null curve, meaning that the holomorphic map $\partial \Phi/\theta=2\partial u/\theta$ maps $M$ into $\boldA_*$.  The space of holomorphic null curves $M\to \C^n$ is denoted $\NC(M, \C^n)$.  The space of real parts of such curves is denoted $\RNC(M, \C^n)$ and $\Re:\NC(M, \C^n)\to \RNC(M, \C^n) \subset \CMI(M,\R^n)$ is the real part map.  As above, we use the subscript $_{\rm nf}$ and the superscript $^{\rm c}$ to denote the corresponding subspaces of nonflat and complete immersions, respectively. It is well known and not hard to see that a holomorphic null curve is complete if and only if its real part is.  The same holds for nonflatness and fullness (defined below).

Theorem \ref{th:main} allows us to strengthen Corollary \ref{cor:weak-eq-for-minimal}.

\begin{corollary}  \label{cor:weak-eq-square}
Let $M$ be an open Riemann surface and $n\geq 3$.  

{\rm (a)}  The inclusions in the square
\[ \xymatrix{
	\RNCnfc(M,\C^n)  \ar@{^{(}->}[r] \ar@{^{(}->}[d]  &  \CMInfc(M,\R^n) \ar@{^{(}->}[d] \\ 
	\RNCnf(M,\C^n)   \ar@{^{(}->}[r]     &  \CMInf(M,\R^n)  } \]
are weak homotopy equivalences. 

{\rm (b)}  If $M$ is of finite topological type, then the inclusions are homotopy equivalences.
\end{corollary}

Part (b) is proved in Section \ref{sec:finite}.  Part (a) is nearly immediate; let us say a few words about the proof.  To prove that the left inclusion is a weak homotopy equivalence, we consider a $P$-family in $\RNCnf(M,\C^n)$ mapping $Q$ into $\RNCnfc(M,\C^n)$ and let all the fluxes in the homotopy vanish: $F^t(p)=0$ for all $p$, $t$.  We do this first for $P$ a singleton and $Q$ empty; then we take $P$ to be the closed unit ball in $\R^k$, $k\geq 1$, and $Q$ to be its boundary sphere.  The right inclusion is handled similarly, ignoring the flux.  For the top inclusion, we take a $P$-family in $\CMInfc(M,\R^n)$ mapping $Q$ into $\RNCnfc(M,\C^n)$ and let the flux homotopy deform the initial flux to zero (we take $K=\varnothing$ and choose, for instance, $F^t(p)=(1-t)\Flux(u_p)$ for all $p$, $t$).  The bottom inclusion is handled in the same way, ignoring completeness.

Theorem \ref{th:main} implies the following analogue of Corollary \ref{co:extension}.

\begin{corollary}
If $M$ is an open Riemann surface and $Q\subset P$ are compact metric spaces such that $Q$ is a retract of $P$, then every continuous map $Q\to\RNCnf(M,\R^n)$, $n\ge 3$, extends to a continuous map $P\to\RNCnf(M,\R^n)$ that takes $P\setminus Q$ into $\RNCnfc(M,\R^n)$. 
\end{corollary}

As noted in \cite{ForstnericLarusson2019CAG}, by continuity in the compact-open topology of the Hilbert transform that takes $u\in \RNCnf(M, \C^n)$ to its harmonic conjugate $v$ with $v(x) = 0$, where $x \in M$ is any chosen base point, the real part map $\Re:\NCnf(M, \C^n)\to \RNCnf(M, \C^n)$ is a homotopy equivalence.  Similarly, $\Re:\NCnfc(M, \C^n)\to \RNCnfc(M, \C^n)$ is a homotopy equivalence.  Corollary \ref{cor:weak-eq-square} therefore implies that the inclusion $\NCnfc(M, \C^n)\hookrightarrow \NCnf(M, \C^n)$ is a weak homotopy equivalence and, if $M$ is of finite topological type, a genuine homotopy equivalence.

It was known previously that the inclusion $\RNCnf(M,\C^n) \hookrightarrow \CMInf(M,\R^n)$ is a weak homotopy equivalence.  It follows from a parametric h-principle for minimal surfaces and holomorphic null curves that was proved in \cite{ForstnericLarusson2019CAG} and used to determine the homotopy type of the spaces of nonflat minimal surfaces in $\R^n$ and nonflat null curves in $\C^n$, $n\geq 3$.  More precisely, it was shown in \cite{ForstnericLarusson2019CAG} that the maps in the diagram
\[ \xymatrix{
\RNCnf(M,\C^n)   \ar@{^{(}->}[r]   &  \CMInf(M,\R^n)  \ar[d]_\psi &  & \\
\NCnf(M,\C^n)  \ar[r]^\phi  \ar[u]^\Re  &  \Onf(M,\boldA_*)  \ar@{^{(}->}[r]  & \Oscr(M,\boldA_*) \ar@{^{(}->}[r] & \Cscr(M,\boldA_*)\\ 
} \]
are weak homotopy equivalences.  Here, $\phi(\Phi)=\partial \Phi/\theta$, $\psi(u)=2\partial u/\theta$, and $\Cscr(M,\boldA_*)$ is the space of continuous maps $M\to\boldA_*$.  When $M$ is of finite topological type, all the maps in the diagram are genuine homotopy equivalences.

Using the above results, we are able to describe the homotopy type of the space of complete nonflat conformal minimal immersions as follows.  The homotopy type of $\Cscr(M,\boldA_*)$ can be understood in terms of basic algebraic topology.

\begin{corollary}   \label{cor:homotopy-type-of-complete}
Let $M$ be an open Riemann surface and $n\geq 3$.  The map
\[ \CMInfc(M,\R^n) \to  \Cscr(M,\boldA_*), \qquad u \mapsto \partial u/\theta, \]
is a weak homotopy equivalence.  When $M$ is of finite topological type, the map is a homotopy equivalence.
\end{corollary}

A conformal minimal immersion $u:M\to \R^n$ is called full if $\psi(u): M\to\boldA_*$ is full in the sense that the $\C$-linear span of $\psi(u)(M)$ is all of $\C^n$.  Similarly, a holomorphic null curve $\Phi:M\to\C^n$ is full if $\phi(\Phi): M\to\boldA_*$ is full.   Fullness and nonflatness are equivalent for $n=3$, but fullness is stronger in higher dimensions.  As we explain in Section \ref{sec:full}, our results are easily adapted to full immersions in place of nonflat immersions.

In conclusion, all the spaces of maps from the open Riemann surface $M$ that we have considered have the same weak homotopy type and, when $M$ is of finite topological type, the same homotopy type.

Further applications of Theorem \ref{th:main} are contained in our subsequent paper \cite{AlarconLarusson2022}, where we use the theorem to, among other results, determine the homotopy type of the space of meromorphic functions on an open Riemann surface $M$ that are the Gauss map of a complete conformal minimal immersion $M\to\R^3$.


\section{A parametric completeness lemma}
\label{sec:lemma}

\noindent
In this section we provide the main step to ensure the completeness condition {\rm (iv)} in Theorem \ref{th:main}. This will be accomplished by a recursive application of the following lemma to the effect of enlarging the boundary distance from a fixed interior point of some of the immersions in a homotopy of nonflat conformal minimal immersions.  Here we only ask that the parameter space $P$ be Hausdorff and compact.  By a compact domain in a topological space we mean a nonempty compact subset which is the closure of a connected open subset.  By a conformal minimal immersion or a holomorphic map on a compact set we mean the restriction of such a map on an unspecified open neighbourhood of the set.
%
%
\begin{lemma}\label{lem:distance}
Let $M$ be an open Riemann surface, $L\subset M$ be a smoothly bounded compact domain, $P$ be a compact Hausdorff space, and $u^t\colon L\times P\to\R^n$ $(t\in [0,1])$, $n\ge 3$, be a homotopy of nonflat conformal minimal immersions $u_p^t:=u^t(\cdot,p)\colon L\to\R^n$ $((p,t)\in P\times[0,1])$.
Also let $Q$ and $T$ be a pair of disjoint closed subspaces of $P$, $K\subset \mathring L$ be a compact subset, and $x_0\in \mathring L$.

Then, for any numbers $\epsilon>0$, $\Lambda>0$, and $0<r<1$, there is a homotopy $\tilde u^t\colon L\times P\to\R^n$ $(t\in[0,1])$ of nonflat conformal minimal immersions $\tilde u_p^t:=\tilde u^t(\cdot,p)\colon L\to\R^n$ $((p,t)\in P\times[0,1])$
satisfying the following conditions.
\begin{enumerate}[\rm (a)]
\item  $\tilde u_p^t=u_p^t$ for all $(p,t)\in (P\times\{0\})\cup (Q\times[0,1])$.

\smallskip
\item $|\tilde u_p^t(x)-u_p^t(x)|<\epsilon$ for all $x\in K$ and $(p,t)\in P\times[0,1]$. 

\smallskip
\item  $\Flux(\tilde u_p^t)=\Flux (u_p^t)$ for all $(p,t)\in P\times [0,1]$. 

\smallskip
\item  $\dist_{\tilde u_p^t}(x_0,bL)>\Lambda$ for all $(p,t)\in T\times[r,1]$.
\end{enumerate}
\end{lemma}
The main point of the lemma is condition {\rm (d)}, which will be the key to guarantee condition {\rm (iv)} in Theorem \ref{th:main}. Except for {\rm (d)}, the initial homotopy $u^t$ itself satisfies the conclusion of the lemma. Here $\dist_{\tilde u_p^t}(\cdot,\cdot)$ denotes the distance function on $L$ induced by the Euclidean distance in $\R^n$ via the immersion $\tilde u_p^t$, that is,
\[
	\dist_{\tilde u_p^t}(x_0,bL)=\inf\{\length(\tilde u_p^t\circ\gamma)\colon
	\text{ $\gamma$ is an arc in $L$ connecting $x_0$ and $bL$}\}, 
\]
where $\length(\cdot)$ denotes the Euclidean length in $\R^n$.

The proof of Lemma \ref{lem:distance} relies on a sort of parametric version of the L\'opez-Ros deformation for minimal surfaces. This deformation, which was introduced in \cite{LopezRos1991JDG} for a different purpose, has proved to be a very powerful tool for the construction of complete minimal surfaces when it is combined with the method by Jorge and Xavier to show the existence of a complete nonflat minimal surface in $\R^3$ contained between two parallel planes \cite{JorgeXavier1980AM}. We refer to \cite[Section 7.1]{AlarconForstnericLopez2021} for background on this subject. The L\'opez-Ros deformation is a way to deform a given conformal minimal immersion on a smoothly bounded compact domain $L$ in an open Riemann surface $M$ while keeping one of its component functions fixed. This was extended to minimal surfaces in $\R^n$ for arbitrary $n\ge 3$ by the following simple trick, first used in \cite{AlarconFernandezLopez2013CVPDE} (see also \cite{AlarconLopez2021APDE,Castro-Infantes2021}). Assume that $u=(u_1,u_2,u_3,\ldots,u_n)\colon L\to\R^n$ is a conformal minimal immersion, let $\theta$ be a nowhere-vanishing holomorphic $1$-form on $M$, and write $2\partial u=(\psi_1,\psi_2,\psi_3,\ldots,\psi_n)\theta$, so $\psi_1^2+\psi_2^2=\Psi:=-\sum_{j=3}^n \psi_j^2$. Setting $f=\psi_1-\imath\psi_2$ and $g=\psi_1+\imath\psi_2$, we have $\psi_1=\frac12(f+g)$, $\psi_2=\frac{\imath}2(f-g)$, and $fg=\Psi$. Multiplying $f$ and dividing $g$ by the same nowhere-vanishing holomorphic function $h$ on $L$, we obtain a pair of holomorphic functions $\tilde \psi_1=\frac12(fh+g/h)$ and $\tilde \psi_2=\frac{\imath}2(fh-g/h)$ such that $\tilde\psi_1^2+\tilde\psi_2^2=fg=\Psi$. Thus, if the $1$-forms $(f-fh)\theta$ and $(g-g/h)\theta$ are exact on $L$, then the formula
\[
	\tilde u(x)=u(x_0)+\Re\int_{x_0}^x(\tilde\psi_1,\tilde\psi_2,\psi_3,\ldots,\psi_n)\theta,
	\quad x\in L,
\]
for any base point $x_0\in \mathring L$, defines a conformal minimal immersion $\tilde u=(\tilde u_1,\tilde u_2,\tilde u_3,\ldots,\tilde u_n)\colon L\to\R^n$ with $\tilde u_j=u_j$ for $j=3,\ldots,n$ and $\Flux(\tilde u)=\Flux(u)$. 

In order to increase the boundary distance of a given immersion $u\colon L\to\R^n$ while hardly modifying it on a given compact subset $K\subset\mathring L$ with $x_0\in \mathring K$, one applies a L\'opez-Ros deformation with a nowhere-vanishing holomorphic function $h$ on $L$ which is close to $1$ on $K$ and large in norm on a Jorge-Xavier-type labyrinth $\Omega$ in $\mathring L\setminus K$  adapted to the given immersion $u$; see e.g. \cite[Section 4]{AlarconFernandezLopez2013CVPDE}. The main difficulty in carrying out this procedure is therefore to find a suitable holomorphic function $h$ on $L$. The following lemma deals with this task in the parametric framework; in fact, it will enable us to enlarge the boundary distance of some of the immersions in a family (see condition {\rm (e)}) while keeping some others fixed (see {\rm (b)}).

%
%
\begin{lemma}\label{lem:pair}
Let $M$ be an open Riemann surface, $L\subset M$ be a smoothly bounded compact domain, $\mathcal D$ be a compact Hausdorff space, and $f,g\colon L\times \mathcal D\to\C$ be a pair of continuous functions such that $f_d:=f(\cdot,d)\colon L\to\C$ and $g_d:=g(\cdot,d)\colon L\to\C$ are holomorphic and complex linearly independent for all $d\in \mathcal D$.
Also let $\theta$ be a nowhere-vanishing holomorphic $1$-form on $M$, $K\subset\mathring L$ be a smoothly bounded compact domain which is a strong deformation retract of $L$, $\Omega\subset \mathring L\setminus K$ be a smoothly bounded $\mathscr O(M)$-convex compact domain, and $\mathcal Y$ and $\mathcal Z$ be disjoint closed subspaces of $\mathcal D$. Then, for any $\epsilon>0$ there is a continuous function $h \colon L\times \mathcal D\to\C^*=\C\setminus\{0\}$ satisfying the following conditions.
\begin{enumerate}[\rm (a)]
\item  The function $h_d:=h(\cdot,d)\colon L\to\C^*$ is holomorphic for all $d\in \mathcal D$.

\smallskip
\item  $h_d=1$ everywhere on $L$ for all $d\in\mathcal Y$.

\smallskip
\item  The $1$-forms $(f_d-f_dh_d)\theta$ and $(g_d-g_d/h_d)\theta$ are exact on $L$ for all $d\in \mathcal D$.

\smallskip
\item  $|h_d(x)-1|<\epsilon$ for all $x\in K$ and $d\in \mathcal D$.

\smallskip
\item  $|h_d(x)|>1/\epsilon$ for all $x\in \Omega$ and $d\in \mathcal Z$.
\end{enumerate}
\end{lemma}

The basic case of Lemma \ref{lem:pair} when $\mathcal D=[0,1]$ can be proved as in \cite{AlarconForstneric2018Crelle} by applying the Oka principle for sections of ramified holomorphic maps with Oka fibres (see \cite{Forstneric2003FM} or \cite[Section 6.13]{Forstneric2017E}), a tool that also enables one to deform conformal minimal immersions in $\R^n$ while keeping some of their component functions fixed (see \cite[Section 3.7]{AlarconForstnericLopez2021}), but is not available in our general parametric framework.

The assumption in Lemma \ref{lem:pair} that the pair of holomorphic functions $f_d$ and $g_d$ be linearly independent for all $d\in \mathcal D$ is used to solve the period problem in condition {\rm (c)}.  A problem with using Lemma \ref{lem:pair} to prove Lemma \ref{lem:distance} is that, for $n\geq 4$, the nonflatness assumption on the immersions $u_p^t$ in Lemma \ref{lem:distance} does not guarantee (even after composing the homotopy $u^t$ by a rigid motion of $\R^n$) that the first and second components of $\partial u_p^t$ are linearly independent for all $(p,t) \in T\times[r,1]$. In order to overcome this difficulty we shall take a suitable finite cover of $T\times[r,1]$ and apply Lemma \ref{lem:pair} in a finite recursive way. 

We defer the proof of Lemma \ref{lem:pair} to later on.

%
%
\begin{proof}[Proof of Lemma \ref{lem:distance} assuming Lemma \ref{lem:pair}]
By possibly enlarging $K$, we may assume that $K$ is a smoothly bounded compact domain which is a strong deformation retract of $L$ and $x_0\in\mathring K$. Also, we assume without loss of generality that $L$ is $\mathscr O(M)$-convex (hence so is $K$); otherwise we  replace $M$ by a regular neighbourhood of $L$ (regularity means that the neighbourhood admits a strong deformation retraction onto $L$). Moreover, for simplicity of exposition we shall assume that $L\setminus \mathring K$ is connected, hence a compact annulus; for the general case it suffices to apply the same procedure in each connected component of $L\setminus \mathring K$.

Let $\theta$ be a holomorphic $1$-form on $M$ vanishing nowhere and set 
\begin{equation}\label{eq:phipt}
	\phi_p^t=(\phi_{p,1}^t,\ldots,\phi_{p,n}^t):=
	\frac{2\partial u_p^t}{\theta}\in \mathscr O(L,\C^n),
	\quad (p,t)\in P\times [0,1];
\end{equation} 
recall that every $u_p^t\colon L\to\R^n$ is a harmonic map. Set 
\[
	\mathcal I:=\{(a,b)\in \{1,\ldots,n\}\times \{1,\ldots,n\}\colon  a<b\}.
\]
Let $(p,t)\in T\times[r,1] \subset (P\setminus Q)\times(0,1]$. Since $u_p^t$ is nonflat, there is $(a,b)\in \mathcal I$ such that the holomorphic $1$-forms $\phi_{p,a}^t$ and $\phi_{p,b}^t$ are complex linearly independent. Since $\phi_p^t$ depends continuously on $(p,t)\in P\times[0,1]$ and $P\times[0,1]$ is a normal topological space, there is a compact neighbourhood $\Upsilon_p^t$ of $(p,t)$ in $P\times [0,1]$ disjoint from $(P\times\{0\})\cup (Q\times[0,1])$ such that 
\begin{equation}\label{eq:Wpt}
	\text{$\phi_{\hat p,a}^{\hat t}$ and $\phi_{\hat p,b}^{\hat t}$ are complex linearly independent for all $(\hat p,\hat t)\in \Upsilon_p^t$.}
\end{equation}
Since $T\times [r,1]\subset \bigcup_{(p,t)\in T\times[r,1]} \mathring \Upsilon_p^t$ is compact, there are finitely many points $(p_l,t_l)\in T\times[r,1]$, $l=1,\ldots,\ell$, such that 
\begin{equation}\label{eq:Wl1}
	T\times [r,1]\subset \bigcup_{l=1}^\ell \mathring \Upsilon_l\subset \bigcup_{l=1}^\ell \Upsilon_l \subset (P\setminus Q)\times (0,1],
\end{equation}
where $\Upsilon_l:=\Upsilon_{p_l}^{t_l}$ for all $l\in\{1,\ldots,\ell\}$. Moreover, condition \eqref{eq:Wpt} ensures the existence of a map $(a,b)\colon \{1,\ldots,\ell\}\to \mathcal I$ such that
\begin{equation}\label{eq:Wl2}
	\text{$\phi_{p,a(l)}^t$ and $\phi_{p,b(l)}^t$ are linearly independent for all 
	$(p,t)\in \Upsilon_l$, $l=1,\ldots,\ell$.}
\end{equation}

Choose a strictly increasing sequence of smoothly bounded $\mathscr O(M)$-convex compact domains 
\begin{equation}\label{eq:K0Kl}
	K_0:=K\subset K_1\subset\cdots \subset K_\ell:=L
\end{equation}
such that $K_{l-1}\subset \mathring K_l$ is a strong deformation retract of $L$ for all $l\in\{1,\ldots,\ell\}$. We may for instance choose $K_l=\{x\in M\colon \varpi(x)\le l/\ell\}$, $l=1,\ldots,\ell-1$, where $\varpi\colon M\to\R$ is a smooth strongly subharmonic Morse exhaustion function such that $K\subset\{x\in M\colon \varpi(x)< 0\}$, $L\supset\{x\in M\colon \varpi(x)\le 1\}$, and $[0,1]$ contains no critical values of $\varpi$; such a function clearly exists by the assumptions on $K$ and $L$ at the very beginning of the proof. In particular, $K_l\setminus \mathring K_{l-1}$ is a smoothly bounded compact annulus (hence connected) for every $l\in\{1,\ldots,\ell\}$.

Set $u^{t,0}:=u^t$ and $\Upsilon_0:=\varnothing$. We shall recursively construct a finite sequence of homotopies $u^{t,l}\colon L\times P\to\R^n$ $(t\in [0,1])$, $l=1,\ldots,\ell$, satisfying the following conditions for all $l\in\{1,\ldots,\ell\}$.
\begin{enumerate}[\rm (A$_l$)]
\item  The map $u_p^{t,l}:=u^{t,l}(\cdot,p)\colon L\to\R^n$ is a nonflat conformal minimal immersion for all $(p,t)\in P\times[0,1]$.

\smallskip
\item  Setting
\[
	\phi_p^{t,l}=(\phi_{p,1}^{t,l},\ldots,\phi_{p,n}^{t,l}):=
	\frac{2\partial u_p^{t,l}}{\theta}\in \mathscr O(L,\C^n),
	\quad (p,t)\in P\times [0,1],
\]
we have that $\phi_{p,a(k)}^{t,l}$ and $\phi_{p,b(k)}^{t,l}$ are complex linearly independent for all $(p,t)\in \Upsilon_k$, $k=1,\ldots,\ell$.

\smallskip
\item  $u_p^{t,l}=u_p^t$ for all $(p,t)\in (P\times\{0\})\cup (Q\times[0,1])$.

\smallskip
\item $| u_p^{t,l}(x)-u_p^{t,l-1}(x)|<\epsilon/\ell$ for all $x\in K_{l-1}$ and $(p,t)\in P\times[0,1]$. 

\smallskip
\item  $\Flux(u_p^{t,l})=\Flux (u_p^t)$ for all $(p,t)\in P\times [0,1]$. 

\smallskip
\item  $\dist_{u_p^{t,l}}(x_0,bK_l)>\Lambda$ for all $(p,t)\in\bigcup_{k=1}^l \Upsilon_k$.
\end{enumerate}

Assuming that such a sequence exists, the homotopy $\tilde u^t:=u^{t,\ell}$ satisfies the conclusion of the lemma. Indeed, each $\tilde u_p^t$ is a nonflat conformal minimal immersion by {\rm (A$_\ell$)}; condition {\rm (a)} equals {\rm (C$_\ell$)}; {\rm (b)} is implied by properties \eqref{eq:K0Kl} and {\rm (D$_l$)}, $l=1,\ldots,\ell$  (recall that $u^t=u^{t,0}$); {\rm (c)} coincides with {\rm (E$_\ell$)}; and {\rm (d)} follows from {\rm (F$_\ell$)}, \eqref{eq:Wl1}, and \eqref{eq:K0Kl}.

To complete the proof it remains to construct the sequence $u^{t,l}$, $l=1,\ldots,\ell$. We proceed by induction. The first step is provided by the already defined homotopy $u^{t,0}=u^t$. Indeed, condition {\rm (A$_0$)} is granted by assumption in the statement of the lemma; {\rm (B$_0$)} is implied by \eqref{eq:phipt} and \eqref{eq:Wl2}; {\rm (C$_0$)} and {\rm (E$_0$)} are obvious by the definition of $u^{t,0}$; and {\rm (D$_0$)} and {\rm (F$_0$)} are empty. For the inductive step, fix $l\in\{1,\ldots,\ell\}$, assume that we have a homotopy $u^{t,l-1}\colon L\times P\to\R^n$ $(t\in[0,1])$ satisfying {\rm (A$_{l-1}$)}--{\rm (F$_{l-1}$)}, and let us furnish such a homotopy $u^{t,l}$ satisfying {\rm (A$_l$)}--{\rm (F$_l$)}.

Write $v_p^t=(v_{p,1}^t,\ldots,v_{p,n}^t)=u_p^{t,l-1}$ and $\psi_p^t=(\psi_{p,1}^t,\ldots,\psi_{p,n}^t)=\phi_p^{t,l-1}$, $(p,t)\in P\times[0,1]$. Also write $a=a(l)$ and $b=b(l)$. Since each immersion $v_p^t$ is conformal we have that $\sum_{j=1}^n (\psi_{p,j}^t)^2=0$, hence
\begin{equation}\label{eq:Psiab}
	(\psi_{p,a}^t)^2+(\psi_{p,b}^t)^2=
	\Psi_p^t:=-\sum_{j\neq a,b}(\psi_{p,j}^t)^2\in \mathscr O(L),\quad (p,t)\in P\times[0,1].
\end{equation}
Condition {\rm (B$_{l-1}$)} ensures that $\Psi_p^t$ is the zero function for no $(p,t)\in \Upsilon_l$, hence, by holomorphicity,  
\begin{equation}\label{eq:Psifinite}
	\text{the zero set of $\Psi_p^t\colon L\to\C$ is finite for all $(p,t)\in \Upsilon_l$.}
\end{equation}

Let $\omega\colon M\to\R$ be a smooth strongly subharmonic Morse exhaustion function such that $K_{l-1}\subset\{x\in M\colon \omega(x)< 0\}$, $\mathring K_l\supset\{x\in M\colon \omega(x)\le 1\}$, and $[0,1]$ contains no critical values of $\omega$; we may for instance choose $\omega$ to be the composition of the already fixed Morse exhaustion function $\varpi$ with a suitable affine transformation. Note that $\omega^{-1}([0,1])\subset \mathring K_l\setminus K_{l-1}$ is a compact annulus. By \eqref{eq:Psifinite}, for each $(p,t)\in \Upsilon_l$ there is $s_p^t\in (0,1)$ such that $\Psi_p^t(x)\neq 0$ for all $x\in \omega^{-1}(s_p^t)$, and hence there is a compact annulus $A_p^t\subset \mathring K_l\setminus K_{l-1}$ of the form $A_p^t=\{x\in M\colon s_p^t\le\omega(x)\le r_p^t\}$, for some $r_p^t\in (s_p^t,1)$ such that $\Psi_p^t(x)\neq 0$ for all $x\in A_p^t$. Since $\Psi_p^t$ depends continuously on $(p,t)\in \Upsilon_l$ there is a compact neighbourhood $W_p^t$ of $(p,t)$ in $P\times[0,1]$ such that $W_p^t\subset (P\setminus Q)\times(0,1]$ and
\begin{equation}\label{eq:Psi-0}
	\Psi_{\hat p}^{\hat t}(x)\neq 0\quad \text{for all $x\in A_p^t$ and $(\hat p,\hat t)\in W_p^t$}.
\end{equation} 
Since $\Upsilon_l\subset \bigcup_{(p,t)\in \Upsilon_l} \mathring W_p^t$ is compact, there are finitely many points $(q_1,c_1),\ldots,$ $(q_m,c_m)$ in $\Upsilon_l$ such that 
\begin{equation}\label{eq:PcupW}
	\Upsilon_l\subset \bigcup_{j=1}^m \mathring W_{q_j}^{c_j}\subset \bigcup_{j=1}^m W_{q_j}^{c_j}\subset (P\setminus Q)\times(0,1].
\end{equation}
  Consider the finitely many annuli $A_{q_j}^{c_j}$, $j=1,\ldots,m$, and, after possibly shrinking each $A_{q_j}^{c_j}$ to a sub-annulus of the form $\{x\in M\colon (s_{q_j}^{c_j})'\le\omega(x)\le (r_{q_j}^{c_j})'\}$ with suitable numbers $s_{q_j}^{c_j}<(s_{q_j}^{c_j})'<(r_{q_j}^{c_j})'<r_{q_j}^{c_j}$, assume that they are pairwise disjoint. Set $W_j:=W_{q_j}^{c_j}$ and $A_j:=A_{q_j}^{c_j}$, $j=1,\ldots,m$.  Fix a number $\varrho>0$ so small that
\begin{equation}\label{eq:varrho}
	|\Psi_p^t(x)|>\varrho\quad \text{for all $x\in A_j$ and $(p,t)\in W_j$, $j=1,\ldots,m$};
\end{equation}
such $\varrho$ exists by \eqref{eq:Psi-0} and compactness of each $A_j$ and each $W_j$.

Since $\theta$ vanishes nowhere on $M$, $|\theta|^2$ is a Riemannian metric on $M$; denote by $\length_\theta(\cdot)$ its associated length function:
\[
	\length_\theta(\gamma):=\int_\gamma|\theta|=\int_0^1|\theta(\gamma(s),\dot\gamma(s))|\, ds\quad
	\text{for every path }\gamma=\gamma(s)\colon [0,1]\to M.
\]
Since each $A_j$ is an annulus, there are a number $\lambda>0$ (small) and smoothly bounded, $\mathscr O(M)$-convex compact domains $\Omega_j\subset \mathring A_j$, $j=1,\ldots,m$, such that the following condition holds for $j=1,\ldots,m$:
\begin{itemize}
\item[\rm ($\star$)]  if $\gamma\colon [0,1]\to A_j$ is a path connecting the two boundary components of $A_j$ and there is no subpath $\tilde\gamma$ of $\gamma$ such that $\tilde\gamma\subset \Omega_j$ and $\length_\theta(\tilde\gamma)>\lambda$, then $\length_\theta(\gamma)>\Lambda/\sqrt\varrho$.
\end{itemize}
Indeed, we can for instance choose each $\Omega_j$ to be a Jorge-Xavier type labyrinth of (finitely many, pairwise disjoint) smoothly bounded closed discs in $\mathring A_j$ (see \cite{JorgeXavier1980AM} or e.g. \cite{AlarconFernandezLopez2012CMH,AlarconFernandezLopez2013CVPDE,AlarconForstneric2018Crelle}) with $\length_\theta(\gamma)>2\Lambda/\sqrt{\varrho}$ for every path $\gamma$ in $A_j\setminus\Omega_j$ connecting the two boundary components of $A_j$, and then take a number $\lambda>0$ sufficiently small. Set 
\[
	\Omega:=\bigcup_{j=1}^m \Omega_j
\]
and note that $K_l\setminus(\mathring K_{l-1}\cup \Omega)$ is path connected since each $\Omega_j$ is $\mathscr O(M)$-convex. 

Let
\[
	f_p^t:=\psi_{p,a}^t-\imath \psi_{p,b}^t \quad\text{and}\quad g_p^t:=\psi_{p,a}^t+\imath \psi_{p,b}^t,
	\quad (p,t)\in P\times[0,1],
\]
and thus define a pair of homotopies $f^t,g^t\colon L\times P\to\C$ $(t\in [0,1])$ with $f_p^t=f^t(\cdot,p)\in\mathscr O(L)$ and $g_p^t=g^t(\cdot,p)\in\mathscr O(L)$ for all $(p,t)\in P\times [0,1]$. Note that
\[
	\psi_{p,a}^t=\frac12 (f_p^t+g_p^t),
	\quad
	\psi_{p,b}^t=\frac{\imath}2 ( f_p^t-g_p^t), 
	\quad\text{and}\quad
	\Psi_p^t=f_p^tg_p^t,
	\quad (p,t)\in P\times[0,1];
\]
see \eqref{eq:Psiab}.
By {\rm (C$_{l-1}$)} and {\rm (B$_{l-1}$)} we have in view of \eqref{eq:phipt} that 
\begin{equation}\label{eq:fp0gp0}
	f_p^t=\phi_{p,a}^t-\imath \phi_{p,b}^t 
	\;\text{and}\; g_p^t=\phi_{p,a}^t+\imath \phi_{p,b}^t
	\; \text{for all $(p,t)\in (P\times\{0\})\cup (Q\times[0,1])$},
\end{equation}
 and $f_p^t$ and $g_p^t$ are complex linearly independent for all $(p,t)\in \Upsilon_l$. Thus, by \eqref{eq:Wl1} and since $P\times[0,1]$ is a normal topological space, there are compact neighbourhoods $\Upsilon$ and $\mathcal D$ of $\Upsilon_l$ in $P\times[0,1]$ such that $\Upsilon\subset\mathring{\mathcal D}\subset\mathcal D\subset (P\setminus Q)\times(0,1]$ and $f_p^t$ and $g_p^t$ are complex linearly independent for all $(p,t)\in \mathcal D$.  Moreover, by \eqref{eq:varrho} there is a number $\sigma>0$ so small that
\begin{equation}\label{eq:sigma}
	|f_p^t(x)|>\sigma\quad \text{for all $x\in \Omega_j$ and $(p,t)\in W_j$, $j=1,\ldots,m$}.
\end{equation} 

Therefore, Lemma \ref{lem:pair} applies with the compact Hausdorff space $\mathcal D$ and the closed subspaces $\mathcal Y=\mathcal D\setminus\mathring\Upsilon$ and $\mathcal Z=\Upsilon_l$, and given $\epsilon_0>0$ to be specified later provides a homotopy $h^t\colon L\times P\to\C^*$ $(t\in [0,1])$ satisfying the following conditions.
\begin{enumerate}[\rm (i)]
\item  The function $h_p^t:=h^t(\cdot,p)\colon L\to\C^*$ is holomorphic for all $(p,t)\in P\times[0,1]$.

\smallskip
\item  $h_p^t=1$ everywhere on $L$ for all $(p,t)\in (P\times [0,1])\setminus \mathring\Upsilon\supset (P\times\{0\})\cup (Q\times[0,1])$.

\smallskip
\item  The holomorphic $1$-forms $(f_p^t-f_p^th_p^t)\theta$ and $(g_p^t-g_p^t/h_p^t)\theta$ are exact on $L$ for all $(p,t)\in P\times[0,1]$.

\smallskip
\item  $|h_p^t(x)-1|<\epsilon_0$ for all $x\in K_{l-1}$ and $(p,t)\in P\times[0,1]$.

\smallskip
\item  $\displaystyle |h_p^t(x)|>1/\epsilon_0>\sqrt{2}\frac{\Lambda}{\lambda\sigma}$ for all $x\in \Omega$ and $(p,t)\in \Upsilon_l$.
\end{enumerate}
We choose $\epsilon_0>0$ so small that the latter inequality in {\rm (v)} is satisfied. Note that Lemma \ref{lem:pair} provides a continuous map $h\colon L\times\mathcal D\to\C^*$ with $h_p^t:=h(\cdot,(p,t))=1$ for all $(p,t)\in\mathcal D\setminus\mathring \Upsilon$; to obtain the homotopy $h^t\colon L\times P\to\C^*$ we just continuously extend $h$ to $L\times P\times[0,1]$ by setting $h_p^t=1$ for all $(p,t)\in (P\times[0,1])\setminus\mathcal D$. 

Set
\[
	\tilde\psi_{p,a}^t:=\frac12\Big(f_p^th_p^t+\frac{g_p^t}{h_p^t}\Big) 
	\quad\text{and}\quad
	\tilde\psi_{p,b}^t:=\frac{\imath}2\Big( f_p^th_p^t-\frac{g_p^t}{h_p^t}\Big),
	\quad (p,t)\in P\times[0,1],
\]
and note that
\begin{equation}\label{eq:tildephi2}
	(\tilde\psi_{p,a}^t)^2+(\tilde\psi_{p,b}^t)^2= f_p^tg_p^t =
	\Psi_p^t, \quad (p,t)\in P\times[0,1],
\end{equation}
and 
\begin{equation}\label{eq:tildephi22}
	|\tilde\psi_{p,a}^t|^2+|\tilde\psi_{p,b}^t|^2
	=\frac12\Big(|f_p^t|^2|h_p^t|^2+\frac{|g_p^t|^2}{|h_p^t|^2}\Big), 
	\quad (p,t)\in P\times[0,1].
\end{equation}
Also note that $\tilde\psi_{p,j}^t$ is close to $\psi_{p,j}^t$ on $K_{l-1}$, $j=a, b$ (depending on $\epsilon_0>0$): see {\rm (iv)}.

By condition {\rm (iii)} the holomorphic $1$-forms $(\tilde \psi_{p,j}^t-\psi_{p,j}^t)\theta$ are exact on $L$ for all $(p,t)\in P\times [0,1]$ and $j=a,b$, and in view of (B${}_{l-1}$) we obtain well-defined homotopies $u_{\cdot,j}^{t,l}\colon L\times P\to\R$ $(t\in[0,1],\, j=a,b)$ of harmonic functions $u_{p,j}^{t,l}\colon L\to\R$ $(p\in P)$ defined by
\[
	u_{p,j}^{t,l}(x)=u_{p,j}^{t,l-1}(x_0)+\Re\int_{x_0}^x \tilde \psi_{p,j}^t\theta,
	\quad x\in L.
\]
Moreover, since each function $h_p^t$ vanishes nowhere, \eqref{eq:tildephi2} ensures that the map $u^{t,l}\colon L\times P\to\R^n$ given by $u^{t,l}(\cdot,p)=(u_{p,j}^{t,l})_{j=1,\ldots,n}$ for all $(p,t)\in P\times[0,1]$, where $u_{p,j}^{t,l}=u_{p,j}^{t,l-1}$ for all $j\notin\{a,b\}$, is a homotopy of conformal minimal immersions $u_p^{t,l}:=u^{t,l}(\cdot,p)\colon L\to\R^n$.

We claim that if $\epsilon_0>0$ is sufficiently small, then the homotopy $u^{t,l}$ satisfies conditions {\rm (A$_l$)}--{\rm (F$_l$)}. Indeed, for such an $\epsilon_0>0$ property {\rm (D$_l$)} follows from {\rm (iv)};  {\rm (A$_l$)} and  {\rm (B$_l$)} are implied by  {\rm (A$_{l-1}$)},  {\rm (B$_{l-1}$)}, and  {\rm (iv)}; {\rm (C$_l$)} is guaranteed by \eqref{eq:phipt}, \eqref{eq:fp0gp0}, and {\rm (ii)}; and {\rm (E$_l$)} follows from {\rm (E$_{l-1}$)} and {\rm (iii)}. Finally, in order to check condition {\rm (F$_l$)} let $(p,t)\in\bigcup_{k=1}^l \Upsilon_k$. If $(p,t)\in \bigcup_{k=1}^{l-1} \Upsilon_k$, then
\[
	\dist_{u_p^{t,l}}(x_0,bK_l) \stackrel{\eqref{eq:K0Kl}}{>} \dist_{u_p^{t,l}}(x_0,bK_{l-1}) 
	\stackrel{\textrm{(iv)}}\approx \dist_{u_p^{t,{l-1}}}(x_0,bK_{l-1}) \stackrel{\textrm{(F$_{l-1}$)}}{>}\Lambda,
\]
hence $\dist_{u_p^{t,l}}(x_0,bK_l)>\Lambda$ provided that $\epsilon_0>0$ is sufficiently small. If $(p,t)\in \Upsilon_l$, let $\gamma$ be a path on $K_l$ connecting $x_0$ and $bK_l$. Take $j\in\{1,\ldots,m\}$ such that $(p,t)\in W_j$ (see \eqref{eq:PcupW}; this $j$ need not be unique) and let $\gamma_j\subset A_j$ be a subpath of $\gamma$ connecting the two boundary components of $A_j$; recall that $x_0\in \mathring K\subset K_{l-1}$. It suffices to check that 
\begin{equation}\label{eq:lengthp1}
	\length(u_p^{t,l}\circ\gamma_j)>\Lambda.
\end{equation}
We distinguish cases. Assume that there is no subpath $\tilde\gamma_j$ of $\gamma_j$ such that $\tilde\gamma_j\subset \Omega_j$ and $\length_\theta(\tilde\gamma_j)>\lambda$. In this case, we have
\[
	\length(u_p^{t,l}\circ\gamma_j) \stackrel{\eqref{eq:tildephi2}}{\ge}  
	\int_{\gamma_j} \sqrt{|\Psi_p^t|}\,|\theta| \stackrel{\eqref{eq:varrho}}{>} 
	 \sqrt\varrho \int_{\gamma_j} |\theta| \stackrel{\textrm{($\star$)}}{>} \Lambda.
\]
If on the contrary there is a subpath $\tilde\gamma_j$ of $\gamma_j$ such that $\tilde\gamma_j\subset \Omega_j$ and $\length_\theta(\tilde\gamma_j)>\lambda$, then 
\[
	\length(u_p^{t,l}\circ\gamma_j) \ge 
	 \int_{\tilde \gamma_j} \sqrt{|\tilde\psi_{p,a}^t|^2+|\tilde\psi_{p,b}^t|^2}\,|\theta|
	   >  \displaystyle \frac{\Lambda}{\lambda}\int_{\tilde \gamma_j} |\theta|  >  \Lambda, 
\]
where in the second to last inequality we have used \eqref{eq:sigma}, \eqref{eq:tildephi22}, and {\rm (v)}.
This shows \eqref{eq:lengthp1} and completes the proof of the lemma granted Lemma \ref{lem:pair}.
\end{proof}
%
%
%
\begin{proof}[Proof of Lemma \ref{lem:pair}]
We assume without loss of generality that $K$ and $L$ are $\mathscr O(M)$-convex; otherwise we replace $M$ by a small regular neighbourhood of $L$. We also assume that $\epsilon<1$ for simplicity of exposition.

Let $\mathscr B=\{C_i\colon i=1,\ldots,l\}$, $l\geq 0$, be a homology basis for $H_1(K,\Z)\cong  \Z^l$ consisting of closed smooth Jordan curves in $\mathring K$ such that 
\begin{equation}\label{eq:C}
	C:=\bigcup_{i=1}^l C_i\subset\mathring K\quad 
	\text{is $\mathscr O(M)$-convex and a  strong deformation retract of $K$}
\end{equation}
and there is a point $x_0\in\mathring K$ such that $C_i\cap C_j=\{x_0\}$ for every pair of distinct indices $i,j\in\{1,\ldots,l\}$. Existence of such a homology basis $\mathscr B$ is well known; see e.g. \cite[Lemma 1.12.10]{AlarconForstnericLopez2021}.
By the assumptions, $\mathscr B$ is a homology basis for $H_1(L, \Z)$ as well. For each $d\in \mathcal D$ consider the period map $\mathcal P_d\colon \mathscr C(C,\C^*)\to(\C^2)^l$ given by
\begin{equation}\label{eq:PeriodMap}
	\mathcal P_d(h)=\left( \int_{C_i}\Big(f_dh \,,\, \frac{g_d}{h}\Big)\theta\right)_{i=1,\ldots,l}\in (\C^2)^l,\quad h\in \mathscr C(C,\C^*).
\end{equation}

The proof of the lemma consists of two independent constructions which are enclosed
in the following two claims.
%
%
\begin{claim}\label{cl:ML-1}
There is a spray of holomorphic functions
\[
	v_\zeta\colon L\to\C^*
	,\quad \zeta\in B,
\]
depending holomorphically on a parameter $\zeta$ in a ball $0\in B\subset\C^N$ for some $N\in\N$, such that 
\begin{equation}\label{eq:v0=1}
	v_0=1
\end{equation} 
and the spray $v_\zeta$ is period dominating in the sense that for each $d\in \mathcal D$, the period map $\tilde{\mathcal P}_d\colon B\to (\C^2)^l$ given by
\begin{equation}\label{eq:tildePeriodMap}
	\tilde{\mathcal P}_d(\zeta)=\mathcal P_d(v_\zeta),\quad \zeta\in B,
\end{equation}
is a submersion at $\zeta=0$.
\end{claim}
Related constructions of period dominating sprays of a multiplicative nature can be found in \cite{AlarconLopez2021APDE} (in a non-parametric framework) and \cite{AlarconForstnericLopez2019JGEA}.  Note that the period domination property of the spray $v_\zeta$ is an open condition which remains valid if we replace the map $(f,g)$ in \eqref{eq:PeriodMap} by any map $(\tilde f,\tilde g)$ in $\mathscr C(C\times\mathcal D,\mathbb C^2)$ sufficiently close to $(f,g)$ uniformly on $C\times \mathcal D$.
%
%
\begin{proof}
Since $f_d$ and $g_d$ are linearly independent for all $d\in \mathcal D$, $(f,g)\colon L\times\mathcal D\to\C^2$ is continuous, and $L\times \mathcal D$ is compact, there are a (large) $k\in\N$ and pairwise distinct points $y_{i,j}\in C_i\setminus\{x_0\}$, $j = 1,\ldots,2k$, $i = 1,\ldots,l$, satisfying the following condition: 
\begin{equation}\label{eq:l.i.}
\begin{array}{c}
	\text{for each $d \in \mathcal D$ there is $j\in \{1, \ldots, k\}$ such that}
	\\
	\big\{ (f_d,g_d)(y_{i,j}) \,,\, (f_d,g_d)(y_{i,k+j})\big\}
	\text{ is a basis of $\C^2$ for all $i=1,\ldots,l$}.
\end{array}
\end{equation}
(Here we are also using the identity principle for the holomorphic functions $f_d$ and $g_d$.) We shall construct a spray $v_\zeta$ of the form
\begin{equation}\label{eq:vzeta}
	v_\zeta=\prod_{i=1}^l\prod_{j=1}^{2k}(1+\zeta_{i,j}a_{i,j}),
\end{equation}
where each $\zeta_{i,j}$ is a complex number and each $a_{i,j}$ is a function in $\mathscr O(L)$ (we write $\zeta=(\zeta^1,\ldots,\zeta^l)\in (\C^{2k})^l$ with $\zeta^i=(\zeta_{i,1},\ldots,\zeta_{i,2k})\in\C^{2k}$, $i=1,\ldots,l$). To perform this task, we shall first construct the functions $a_{i,j}$ as continuous functions in $\mathscr C(C,\C)$ and then upgrade them to holomorphic functions in $\mathscr O(L)$ by Mergelyan approximation, as we may in view of \eqref{eq:C}.  Clearly, \eqref{eq:v0=1} holds.

For each $i\in\{1,\ldots,l\}$, let $\gamma_i\colon (0,1)\to C_i$ be a smooth parametrisation of $C_i\setminus\{x_0\}$ and extend $\gamma_i$ continuously to $[0,1]$ with $\gamma_i(0)=\gamma_i(1)=x_0$. For each $j\in\{1,\ldots,2k\}$ let $s_{i,j}\in (0,1)$ be the only point with $\gamma_i(s_{i,j})=y_{i,j}$ and choose a number $\tau>0$ to be specified later, so small that $0<s_{i,j}-\tau<s_{i,j}+\tau<1$ for all $i,j$ and the intervals $[s_{i,j}-\tau,s_{i,j}+\tau]$, $j=1,\ldots,2k$, are pairwise disjoint for all $i=1,\ldots,l$. Next, for each $i,j$ take a continuous function $a_{i,j}\colon C_i\to\C$ such that
\begin{equation}\label{eq:aij1}
	a_{i,j}(\gamma_i(s))=0\quad \text{for all }s\in[0,1]\setminus[s_{i,j}-\tau,s_{i,j}+\tau]
\end{equation}
(hence $a_{i,j}(x_0)=0$ for all $i,j$) and
\begin{equation}\label{eq:aij2}
	\int_{C_i} a_{i,j}\theta=
	\int_{s_{i,j}-\tau}^{s_{i,j}+\tau} a_{i,j}(\gamma_i(s))\, \theta(\gamma_i(s),\dot\gamma_i(s))\, ds=1.
\end{equation}
Extend each $a_{i,j}$ continuously to $C$ by setting $a_{i,j}=0$ on $C\setminus C_i$, consider the continuous function $v_\zeta\colon C\to\C$ defined by the expression in \eqref{eq:vzeta},
and assume that the ball $0\in B\subset \C^{2kl}$ is so small that $v_\zeta$ vanishes nowhere on $C$ for all $\zeta\in B$. We have that $v_\zeta\colon C\to\C^*$ depends holomorphically on $\zeta$. 
Observe that
\[
	\frac{\partial v_\zeta(x)}{\partial \zeta_{i,j}}\Big|_{\zeta=0}=a_{i,j}(x),
	\quad x\in C,\; i\in\{1,\ldots,l\}, \; j\in\{1,\ldots,2k\},
\]
hence, in view of \eqref{eq:PeriodMap}, \eqref{eq:tildePeriodMap}, \eqref{eq:aij1}, and \eqref{eq:aij2},
for any sufficiently small choice of $\tau>0$ we have for each $d\in \mathcal D$, $i\in \{1,\ldots,l\}$, and $j\in\{1,\ldots,2k\}$ that
\begin{eqnarray*}
	\frac{\partial \tilde{\mathcal P}_d(\zeta)}{\partial \zeta_{i,j}}\Big|_{\zeta=0} 
	& = &
	\left( \int_{C_m} ( f_d \,,\, -g_d )a_{i,j}\,\theta \right)_{m=1,\ldots,l}
	\\
	& \approx & \big((f_d(y_{i,j}) \,,\, -g_d(y_{i,j}))\delta_{im}\big)_{m=1,\ldots,l}\in (\C^2)^l, 
\end{eqnarray*}
where $\delta_{im}$ is the Kronecker delta and the smaller $\tau > 0$, the closer the approximation. Thus, in view of \eqref{eq:l.i.} we obtain that
\begin{equation}\label{eq:submersion}
	\frac{\partial \tilde{\mathcal P}_d(\zeta)}{\partial \zeta}\Big|_{\zeta=0}
	\colon T_0 B\cong \C^{2kl}\to (\C^2)^l\quad
	\text{is surjective for all $d\in \mathcal D$}
\end{equation}
provided that $\tau>0$ has been chosen sufficiently small.
As we mentioned above, to conclude the proof of the claim it now suffices to approximate each function $a_{i,j}$ uniformly on $C$ by a function in $\mathscr O(L)$ (with the same name); this is granted by the classical Mergelyan theorem \cite{Bishop1958PJM} in view of \eqref{eq:C}. If all these approximations are close enough, then \eqref{eq:submersion} guarantees the period domination condition of $v_\zeta$ in the statement of the claim. After shrinking the ball $B$ to ensure that $v_\zeta$ vanishes nowhere on $L$ for all $\zeta \in B$, this concludes the proof.
\end{proof}

%
%
\begin{claim}\label{cl:ML-2}
For any number $0<\mu<1$, there is a continuous function $w\colon L\times \mathcal D\to\C^*$ satisfying the following conditions.
\begin{enumerate}[\rm (i)]
\item  The function $w_d:=w(\cdot,d)\colon L\to\C^*$ is holomorphic for all $d\in \mathcal D$.

\smallskip
\item  $w_d=1$ everywhere on $L$ for all $d\in \mathcal Y$.

\smallskip
\item  $|w_d(x)-1|<\mu$ for all $x\in K$ and $d\in \mathcal D$.

\smallskip
\item  $|w_d(x)|>1/\mu$ for all $x\in \Omega$ and $d\in \mathcal Z$.
\end{enumerate}
\end{claim}
%
%
\begin{proof}
Since $\mathcal D$ is compact and Hausdorff, it is a normal topological space, and since $\mathcal Y$ and $\mathcal Z$ are disjoint closed subspaces of $\mathcal D$, Urysohn's lemma yields a continuous function $\Phi\colon \mathcal D\to [0,1]$ such that
\begin{equation}\label{eq:Urysohn}
	\text{$\Phi(d)=0$ for all $d\in \mathcal Y$\quad and\quad $\Phi(d)=1$ for all $d\in \mathcal Z$}.
\end{equation}

Let $K'$ and $\Omega'$ be a pair of disjoint smoothly bounded compact domains such that $K\subset \mathring K'$, $\Omega\subset \mathring \Omega'$, and $K'\cup\Omega'$ is $\mathscr O(M)$-convex.
Consider the function $\tilde w\colon (K'\cup \Omega')\times \mathcal D\to\C^*$ determined by the locally constant (hence holomorphic) functions $\tilde w_d:=\tilde w(\cdot,d)\colon K'\cup \Omega'\to [1,+\infty)\subset \C^*$ $(d\in \mathcal D)$ given by
\[
	\tilde w_d(x)=\left\{\begin{array}{ll}
	1 & x\in K'
	\\
	\displaystyle 1+\frac{\Phi(d)}{\mu} & x\in\Omega',\end{array}\right.
	\quad  d\in \mathcal D.
\]
In view of \eqref{eq:Urysohn}, we have that 
\[
	\text{$\tilde w_d=1$ for all $d\in \mathcal Y$\quad and\quad 
	$\tilde w_d(x)>1/\mu$ for all $(x,d)\in \Omega'\times\mathcal Z$.}
\]
Since $L\times \mathcal D$ is a normal topological space and $(L\times\mathcal Y)\cup (\Omega'\times\mathcal Z)$ is a closed subset, the Tietze extension theorem implies that $\tilde w$ extends to a continuous function $\tilde w\colon L\times \mathcal D\to (0,+\infty)\subset \C^*$ such that $\tilde w(\cdot,d)=1$ for all $d\in \mathcal Y$. Therefore, since $K'\cup\Omega'$ is $\mathscr O(M)$-convex and contains $K\cup\Omega$ in its interior, the parametric Oka property with approximation for holomorphic functions into $\C^*$ (see \cite[Theorem 5.4.4]{Forstneric2017E} and recall that $\C^*$ is Oka) enables us to approximate $\tilde w$ uniformly on $(K\cup \Omega)\times \mathcal D$ by a function $w\colon L\times \mathcal D\to\C^*$ satisfying the conclusion of the claim.
\end{proof}

With the above two claims in hand, the proof of Lemma \ref{lem:pair} is completed as follows. Fix a number 
\begin{equation}\label{eq:lambdaepsilon}
	0<\lambda<\frac{\epsilon}{3}
\end{equation}
and, by \eqref{eq:v0=1} and after shrinking the ball $B$ if necessary, assume that
\begin{equation}\label{eq:vzeta1}
	|v_\zeta(x)-1|<\lambda\quad \text{for all $x\in L$ and $\zeta\in B$},
\end{equation}
where $v_\zeta$ is the spray provided by Claim \ref{cl:ML-1}. 
Fix another number 
\begin{equation}\label{eq:muepsilon}
	0<\mu<\frac{\epsilon}{3}
\end{equation}
 to be specified later, let $w\colon L\times \mathcal D\to\C^*$ be a function given by Claim \ref{cl:ML-2} for the fixed number $\mu$, and define
\[
	\tilde h_{d,\zeta}:= w_d v_\zeta\colon L\to\C^*,\quad d\in \mathcal D, \;  \zeta\in B.
\]
The function $\tilde h_{d,\zeta}$ is holomorphic and depends continuously on $d\in \mathcal D$ and holomorphically on $\zeta\in B$. We have by \eqref{eq:v0=1} that 
\[
	\tilde h_{d,0}=w_d\quad \text{for all $d\in \mathcal D$};
\]
together with condition {\rm (ii)} we infer that
\begin{equation}\label{eq:tildeh1}
	\text{$\tilde h_{d,0}=1$ everywhere on $L$ 
	for all $d\in \mathcal Y$.}
\end{equation}
Moreover, \eqref{eq:vzeta1} and conditions {\rm (iii)} and {\rm (iv)} ensure that
\begin{equation}\label{eq:tildeh2}
	|\tilde h_{d,\zeta}(x)-1|<(1+\mu)\lambda+\mu\quad 
	\text{for all $x\in K$, $d\in \mathcal D$, and $\zeta\in B$}
\end{equation}
and
\begin{equation}\label{eq:tildeh3}
	|\tilde h_{d,\zeta}(x)|>\frac{1-\lambda}{\mu}
	\quad \text{for all $x\in \Omega$, $d\in \mathcal Z$, and $\zeta\in B$}.
\end{equation}
By \eqref{eq:v0=1}, \eqref{eq:tildeh1}, Claim \ref{cl:ML-2} (iii), and the fact that $C\subset\mathring K$ (see \eqref{eq:C}), the period domination property of the spray $v_\zeta$ guarantees that for any sufficiently small choice of $\mu>0$, the implicit function theorem gives a continuous map
\[
	\zeta\colon \mathcal D\to B\subset\C^N
\]
such that 
\begin{equation}\label{eq:zeta0}
	\zeta(d)=0\quad \text{for all $d\in \mathcal Y$}
\end{equation}
and the function
\[
	h_d:=\tilde h_{d,\zeta(d)}\colon L\to\C^*,\quad d\in \mathcal D
\]
satisfies 
\begin{equation}\label{eq:hptperiods}
	\mathcal P_d(h_d)=\mathcal P_d(1)\quad \text{for all $d\in \mathcal D$}.
\end{equation}
Indeed, we are using here that for sufficiently small $\mu>0$ the spray $v_\zeta$ is period dominating with respect to the period map $B\to (\C^2)^l$ given by
\[
	B\ni\zeta\longmapsto \left( \int_{C_i}\Big((f_dw_d)v_\zeta \,,\, \frac{(g_d/w_d)}{v_\zeta}\Big)\theta\right)_{i=1,\ldots,l}\in (\C^2)^l
\] 
for every $d\in\mathcal D$; see the remark below the statement of Claim \ref{cl:ML-1}.

We claim that the continuous function $h\colon L\times \mathcal D\to \C^*$  determined by $h(\cdot,d):=h_d$ for all $d\in \mathcal D$ satisfies the conclusion of Lemma \ref{lem:pair}. Indeed, condition {\rm (a)} is already seen; {\rm (b)} is guaranteed by \eqref{eq:tildeh1} and \eqref{eq:zeta0}; {\rm (c)} is implied by \eqref{eq:hptperiods}, \eqref{eq:C}, and \eqref{eq:PeriodMap}; {\rm (d)} is ensured by \eqref{eq:tildeh2}, \eqref{eq:lambdaepsilon}, and \eqref{eq:muepsilon}; and {\rm (e)} follows from \eqref{eq:tildeh3}, \eqref{eq:lambdaepsilon}, and \eqref{eq:muepsilon} (take into account that $0<\epsilon<1)$.
\end{proof}
%
%

Lemma \ref{lem:distance} is proved.


\section{Prescribing the flux}
\label{sec:flux}

\noindent
In this section we generalise the methods in \cite{ForstnericLarusson2019CAG} to control the periods not just of the immersions $u_p^1$ but of all the immersions $u_p^t$ in the homotopy, under the appropriate assumptions.
%
%
\begin{lemma}\label{lem:flux}
Let $M$ be an open Riemann surface and $K$ and $L$ be a pair of smoothly bounded $\mathscr O(M)$-convex compact domains in $M$ such that $K\subset \mathring L$ and the Euler characteristic of $L\setminus\mathring K$ equals $0$ or $-1$.  Let $Q\subset P$ be compact Hausdorff spaces, let $u^t\colon K\times P\to\R^n$ $(t\in [0,1])$, $n\ge 3$, be a homotopy of nonflat conformal minimal immersions $u_p^t:=u^t(\cdot,p)\colon K\to\R^n$, and let $F^t\colon P\to H^1(L,\R^n)$ $(t\in [0,1])$ be a homotopy of cohomology classes $F_p^t:=F^t(p)$ satisfying the following conditions.
\begin{enumerate}[\rm (I)]
\item  $u_p^t=u_p^0$ for all $(p,t)\in Q\times[0,1]$.
\smallskip
\item  $u_p^0$ extends to a conformal minimal immersion $u_p^0\colon L\to\R^n$ for all $p\in P$.
\smallskip
\item  $F_p^t=\Flux(u_p^t)$ for all $(p,t)\in (P\times\{0\})\cup (Q\times[0,1])$.
\smallskip
\item  $F_p^t|_K=\Flux(u_p^t)$ for all $(p,t)\in P\times[0,1]$.
\end{enumerate}
Then, for any $\epsilon>0$ there is a homotopy $\tilde u^t\colon L\times P\to\R^n$ $(t\in [0,1])$ of nonflat conformal minimal immersions $\tilde u_p^t:=\tilde u^t(\cdot,p)\colon L\to\R^n$ satisfying the following conditions.
\begin{enumerate}[\rm (i)]
\item  $\tilde u_p^t=u_p^0$ for all $(p,t)\in (P\times\{0\})\cup(Q\times[0,1])$.
\smallskip
\item  $|\tilde u_p^t(x)-u_p^t(x)|<\epsilon$ for all $x\in K$ and $(p,t)\in P\times[0,1]$.
\smallskip
\item  $\Flux(\tilde u_p^t)=F_p^t$ for all $(p,t)\in P\times[0,1]$.
\end{enumerate}
\end{lemma}

The proof of Lemma \ref{lem:flux} consists of adapting the arguments in the proof of Theorem 4.1 in \cite{ForstnericLarusson2019CAG} by using \cite[Lemma 3.1]{ForstnericLarusson2019CAG} in its full generality.
%
%
\begin{proof}
If the Euler characteristic of $L\setminus\mathring K$ equals $0$, then $L\setminus \mathring K$ is a union of finitely many, pairwise disjoint compact annuli. Thus, $K$ is a strong deformation retract of $L$ and the inclusion $K\hookrightarrow L$ induces an isomorphism $H^1(K,\R^n)\to H^1(L,\R^n)$.  With the identification given by this isomorphism, condition {\rm (IV)} says that 
\begin{equation}\label{eq:piKL}
	F_p^t=\Flux(u_p^t)\quad \text{for all $(p,t)\in P\times[0,1]$}.
\end{equation} 
In this case, the result follows by an inspection of the proof of \cite[Theorem 4.1]{ForstnericLarusson2019CAG}. Indeed, our situation corresponds to the noncritical case in that proof except that we do not have the assumptions {\rm (b')} and {\rm (c')} there (see \cite[p.\ 21]{ForstnericLarusson2019CAG}). Following the argument in that proof but without paying attention to some immersions in the family having vanishing flux, we obtain a homotopy of nonflat conformal minimal immersions $\tilde u_p^t\colon L\to\R^n$, $(p,t)\in P\times[0,1]$, satisfying {\rm (i)}, {\rm (ii)}, and $\Flux(\tilde u_p^t|_K)=\Flux(u_p^t)$ for all $(p,t)\in P\times[0,1]$ (cf.\ conditions {\rm ($\alpha$)}, {\rm ($\beta$)}, and {\rm ($\gamma$)} in \cite[p.\ 22]{ForstnericLarusson2019CAG}). The latter and \eqref{eq:piKL} imply {\rm (iii)}, thereby concluding the proof in this case.

Assume now that the Euler characteristic of $L\setminus\mathring K$ equals $-1$. In this case $L\setminus\mathring K$ is a disjoint union of finitely many compact annuli and a single pair of pants (that is, a sphere from which three smoothly bounded open discs with pairwise disjoint closures have been removed). Thus, $L$ admits a strong deformation retraction onto a compact set $S=K\cup E$, where $E$ is an embedded arc in $\mathring L\setminus\mathring K$ with its two endpoints in $K$ and otherwise disjoint from $K$.  The arc $E$ lies in the pair of pants. We choose $S$, as we may, to be an admissible subset of $M$ in the sense of \cite[Definition 2.1]{ForstnericLarusson2019CAG}.

Let $\theta$ be a holomorphic $1$-form on $M$ vanishing nowhere and set
\[
	f_p^t:=\frac{2\partial u_p^t}{\theta}\Big|_S\colon S\to \boldA_*,
	\quad (p,t)\in (P\times\{0\})\cup (Q\times[0,1]).
\]
Note that $f_p^t=f_p^0$ for all $(p,t)\in Q\times[0,1]$. Also set
\[
	f_p^t:=\frac{2\partial u_p^t}{\theta}\colon K\to \boldA_*,
	\quad (p,t)\in (P\setminus Q)\times (0,1].
\]
We claim that there are continuous families of smooth maps 
\[
	g_p^t\colon S\to \boldA_*,\quad v_p^t\colon S\to\R^n,\quad (p,t)\in P\times[0,1],
\]
satisfying the following conditions.
\begin{enumerate}[\rm (a)]
\item  $v_p^t|_K=u_p^t$ and $g_p^t|_K=f_p^t$ for all $(p,t)\in P\times[0,1]$.
\smallskip
\item  $v_p^t=u_p^0|_S$ and $g_p^t=f_p^0|_S$ for all $(p,t)\in (P\times \{0\})\cup (Q\times[0,1])$.
\smallskip
\item  The pair $U_p^t=(v_p^t,g_p^t\theta)$ is a nonflat generalised conformal minimal immersion on $S$ in the sense of \cite[Definition 2.2]{ForstnericLarusson2019CAG}.
\smallskip
\item  $\Flux(U_p^t)=F_p^t$ for all $(p,t)\in P\times[0,1]$.
\end{enumerate}
(Cf.\ conditions {\rm ($\mathfrak a$)}--{\rm ($\mathfrak d$)} in \cite[p.\ 26]{ForstnericLarusson2019CAG}; in particular, compare {\rm (d)} here with {\rm ($\mathfrak d$)} there.) 
Indeed, extend $E$ to a real-analytic  Jordan curve $C\subset \mathring L$ with $C\setminus \mathring K=E$. Set $C_3:=C\cap K$, take a real-analytic parametrisation $\gamma\colon [0,3]\to C$ such that $\gamma([2,3])=C_3$, and set $C_i:=\gamma([i-1,i])$ for $i=1,2$; hence $C=C_1\cup C_2\cup C_3$. Extend the maps $f_p^t\colon K\to\boldA_*$ for $(p,t)\in (P\setminus Q)\times (0,1]$ continuously to $S=K\cup C$ so that the family $f_p^t\colon S\to\boldA_*$, $(p,t)\in P\times[0,1]$, depends continuously on $(p,t)$; in particular, the extension is the already defined map $f_p^t$ on $S$ for all $(p,t)\in (P\times\{0\})\cup (Q\times[0,1])$. Choose $0<\eta<1/2$ (small) and set
\[
	I_i=[i-1+\eta,i-\eta] \quad\text{and}\quad C_i'=\gamma(I_i),\quad i=1,2.
\]
We choose $f_p^t$ such that $f_p^t|_{C_1'\cup C_2'}=f_p^0|_{C_1'\cup C_2'}$ for all $(p,t)\in P\times [0,1]$. Define $\sigma_p^t\colon [0,3]\to \boldA_*$, $(p,t)\in P\times[0,1]$, by
\[
	\sigma_p^t(s)=f_p^t(\gamma(s))\, \theta(\gamma(s),\dot\gamma(s)),\quad s\in [0,3].
\]
Note that $\sigma_p^t=\sigma_p^0$ for all $p\in Q$ and $\int_0^3\sigma_p^t(s)\, ds=F_p^t([C])=F_p^0([C])$ for all $(p,t)\in (P\times\{0\})\cup (Q\times[0,1])$; see assumptions {\rm (I)} and {\rm (III)}. Thus, for any (small) $\delta>0$ Lemma 3.1 in \cite{ForstnericLarusson2019CAG} furnishes us with a continuous family of paths $\tilde \sigma_p^t\colon [0,1]\to \boldA_*$, $(p,t)\in P\times [0,1]$, satisfying the following conditions.
\begin{enumerate}[\rm ({A}1)]
\item  $\tilde\sigma_p^t=\sigma_p^t$ on $[0,1]\setminus I_1$ for all $(p,t)\in P\times[0,1]$.
\smallskip
\item  $\tilde \sigma_p^t=\sigma_p^0|_{[0,1]}$ for all $(p,t)\in (P\times\{0\})\cup (Q\times[0,1])$.
\smallskip
\item  $\displaystyle\Big|\int_0^1\tilde \sigma_p^t(s)\, ds+\int_1^3 \sigma_p^t(s)\, ds-F_p^t([C])\Big|<\delta$ for all $(p,t)\in P\times[0,1]$. 
\end{enumerate}
(Cf. \cite[Eq.\ (4.11)]{ForstnericLarusson2019CAG}; this is the precise point at which we take advantage of the full generality of \cite[Lemma 3.1]{ForstnericLarusson2019CAG}.) Next, arguing as in \cite[p.\ 28--29]{ForstnericLarusson2019CAG}, assuming that $\delta>0$ is sufficiently small we can find a continuous family of paths $\tilde \sigma_p^t\colon [1,2]\to \boldA_*$, $(p,t)\in P\times [0,1]$, satisfying the following conditions.
\begin{enumerate}[\rm ({B}1)]
\item  $\tilde\sigma_p^t=\sigma_p^t$ on $[1,2]\setminus I_2$ for all $(p,t)\in P\times[0,1]$.
\smallskip
\item  $\tilde \sigma_p^t=\sigma_p^0|_{[1,2]}$ for all $(p,t)\in (P\times\{0\})\cup (Q\times[0,1])$.
\smallskip
\item  $\displaystyle\int_0^2\tilde \sigma_p^t(s)\, ds+\int_2^3 \sigma_p^t(s)\, ds=F_p^t([C])$ for all $(p,t)\in P\times[0,1]$.
\end{enumerate}
(Cf. \cite[Eq.\ (4.12) and (4.13)]{ForstnericLarusson2019CAG}.) Define $g_p^t\colon S\to\boldA_*$ and $v_p^t\colon S\to\R^n$, $(p,t)\in P\times[0,1]$, by
\[
	g_p^t|_K=f_p^t|_K\quad\text{and}\quad
	g_p^t(\gamma(s))=\frac{\tilde\sigma_p^t(s)}{\theta(\gamma(s),\dot\gamma(s))}
	\quad\text{for all }s\in [0,2],
\]
and
\[
	v_p^t|_K=u_p^t\quad\text{and}\quad
	v_p^t(\gamma(s))=u_p^t(\gamma(0))+\int_0^s\tilde \sigma_p^t(\varsigma)\, d\varsigma
	\quad\text{for all }s\in [0,2].
\]
Properties {\rm (A1)}--{\rm (A3)} and {\rm (B1)}--{\rm (B3)} trivially show that $g_p^t$ and $v_p^t$ satisfy conditions {\rm (a)}--{\rm (d)}. Arguing as in \cite[p.\ 26--27]{ForstnericLarusson2019CAG}, this reduces the proof to the case of Euler characteristic equal to $0$. This completes the proof of the lemma.
\end{proof}


\section{Proof of Theorem \ref{th:main}}
\label{sec:proof}

\noindent
The proof consists of a standard recursive process using Lemmas \ref{lem:distance} and \ref{lem:flux}; the former will enable us to ensure the completeness of the immersions in the limit homotopy while the latter will allow us to control their fluxes.

Let $K\subset M$ and $Q\subset P$ be as in the statement of the theorem. Without loss of generality, we may assume that $K$ is a smoothly bounded $\Oscr(M)$-convex compact domain.  Since $P$ is a compact metric space and $Q\subset P$ is a closed subspace, there is a sequence of closed subspaces $T_j\subset P$, $j\in\N=\{1,2,3,\ldots\}$, such that
\begin{equation}\label{eq:Tj}
	T_j\subset\mathring T_{j+1}\;\text{ for all $j\in\N$}
	\quad\text{and}\quad 
	\bigcup_{j\in\N} T_j=P\setminus Q.
\end{equation}

It is only here that it is not sufficient to assume that $P$ is a compact Hausdorff space.  Such a space is normal, but we need $P$ to be perfectly normal in order to guarantee the existence of the subspaces $T_j$.  We have opted to impose the simple sufficient condition that $P$ be metrisable.  This is a harmless assumption since a family can always be reparametrised by its image and our families take their values in metrisable spaces.

Set $K_0:=K$ and take a sequence of smoothly bounded $\mathscr O(M)$-convex compact domains $K_j$ in $M$, $j\in\N$, such that 
\begin{equation}\label{eq:Kj}
	K_{j-1}\subset\mathring K_j\;\text{ for all $j\in\N$}, \quad
	\bigcup_{j\in\N} K_j=M,
\end{equation}
and
\begin{equation}\label{eq:Kj-Euler}
	\text{the Euler characteristic of $K_j\setminus \mathring K_{j-1}$ equals $0$ or $-1$ for all $j\in \N$}.
\end{equation}
Existence of such a sequence is well known; see e.g. \cite[Lemma 4.2]{AlarconLopez2013JGEA}.  Set 
\begin{equation}\label{eq:upt0K0}
	u_p^{t,0}:=u_p|_{K_0}\in\CMInf(K_0,\R^n),\quad (p,t)\in P\times[0,1].
\end{equation}

Let $\epsilon>0$ and let $F^t\colon P\to H^1(M,\R^n)$ $(t\in [0,1])$ be a homotopy of cohomology classes $F^t_p:=F^t(p)$ as in the statement of the theorem. Fix $x_0\in\mathring K=\mathring K_0$ and set $T_0:=\varnothing$, $\epsilon_0:=\epsilon$, $\epsilon_{-1}:=3\epsilon$, and $K_{-1}:=\varnothing$. We shall recursively construct a sequence of numbers $\epsilon_j>0$, $j\in\N$, and a sequence of homotopies $u^{t,j}\colon K_j\times P\to\R^n$ $(t\in [0,1])$ of nonflat conformal minimal immersions
\[
	u_p^{t,j}:=u^{t,j}(\cdot,p)\colon K_j\to\R^n,\quad (p,t)\in P\times[0,1],\; j\in \N,
\]
such that the following conditions are satisfied for all $j\in\N$.
\begin{enumerate}[\rm (A$_j$)]
\item  $u_p^{t,j}=u_p|_{K_j}$ for all $(p,t)\in (P\times\{0\})\cup (Q\times[0,1])$.

\smallskip
\item  $|u_p^{t,j}(x)-u_p^{t,j-1}(x)|<\epsilon_j$ for all $x\in K_{j-1}$ and $(p,t)\in P\times[0,1]$.

\smallskip
\item  $\dist_{u_p^{t,j}}(x_0,bK_j)>j$ for all $(p,t)\in T_j\times[\frac1{j+1},1]$.

\smallskip
\item  $\epsilon_j<\epsilon_{j-1}/2$ and if $u\colon M\to\R^n$ is a conformal harmonic map such that $|u(x)-u_p^{t,j-1}(x)|<2\epsilon_j$ for all $x\in K_{j-1}$ and some $(p,t)\in P\times[0,1]$, then $u|_{K_{j-1}}$ is a nonflat immersion. Moreover, if $|u(x)-u_p^{t,j-1}(x)|<2\epsilon_j$ for all $x\in K_{j-1}$ and some $(p,t)\in T_{j-1}\times[\frac1{j},1]$, then $\dist_{u}(x_0,bK_{j-1})>j-1$.

\smallskip
\item $\Flux(u_p^{t,j})=F_p^t|_{K_j}$ for all $(p,t)\in P\times[0,1]$.
\end{enumerate}

Assuming that such sequences exist, conditions {\rm (B$_j$)}, {\rm (D$_j$)}, and \eqref{eq:Kj} ensure that there is a limit homotopy
\[
	u_p^t:=\lim_{j\to\infty} u_p^{t,j}\colon M\to\R^n,\quad (p,t)\in P\times[0,1],
\]
such that
\begin{equation}\label{eq:Cauchy}
	|u_p^t(x)-u_p^{t,j-1}(x)|<2\epsilon _j\quad \text{for all $x\in K_{j-1}$ and $(p,t)\in P\times[0,1]$, $j\in\N$}.
\end{equation}
We claim that the homotopy $u^t\colon M\times P\to\R^n$ $(t\in [0,1])$ given by $u^t(\cdot,p):=u_p^t$ for all $(p,t)\in P\times[0,1]$ satisfies the conclusion of the theorem. Indeed, conditions {\rm (i)} and {\rm (iii)} are implied by \eqref{eq:Cauchy} and {\rm (D$_j$)} (recall that $\epsilon=\epsilon_0$); {\rm (ii)} is ensured by {\rm (A$_j$)}; and {\rm (v)} is guaranteed by {\rm (E$_j$)}. Finally, in order to check {\rm (iv)} let $(p,t)\in (P\setminus Q)\times(0,1]$. By \eqref{eq:Tj} there is a large enough $j_0\in\N$ such that $(p,t)\in T_{j-1}\times[\tfrac1{j},1]$ for all $j\ge j_0$. Therefore, \eqref{eq:Cauchy} and {\rm (D$_j$)} guarantee that  $\dist_{u_p^t}(x_0,bK_{j-1})>j-1$ for all $j>j_0$; hence, in view of \eqref{eq:Kj},  $u_p^t$ is complete.

It remains to construct the sequences; we proceed by induction. For the first step, note that condition {\rm (A$_0$)} is given by \eqref{eq:upt0K0}; {\rm (B$_0$)} and {\rm (D$_0$)} are empty (we take $K_{-1}:=\varnothing$ and, for instance, $\epsilon_{-1}:=3\epsilon$); {\rm (C$_0$)} follows from the facts that $x_0\in \mathring K_0$ and each map $u_p^{t,0}$ is an immersion on $K_0$; and {\rm (E$_0$)} is granted by \eqref{eq:upt0K0} and  the assumption in the statement of the theorem. For the inductive step, fix $j\in\N$, assume that we have $\epsilon_{j-1}>0$ and a homotopy $u^{t,j-1}\colon K_{j-1}\to\R^n$ $(t\in [0,1])$ satisfying {\rm (A$_{j-1}$)}--{\rm (E$_{j-1}$)}, and let us provide a number $\epsilon_j>0$ and a homotopy $u^{t,j}$ satisfying  conditions {\rm (A$_j$)}--{\rm (E$_j$)}. 

In view of {\rm (C$_{j-1}$)} there is a number $\epsilon_j>0$ satisfying {\rm (D$_j$)}; use the Cauchy estimates and see \cite[Section 2]{AlarconForstneric2019RMI}.  By \eqref{eq:Kj}, \eqref{eq:Kj-Euler}, and {\rm (A$_{j-1}$)}, Lemma \ref{lem:flux} applies with $K_j$ and $K_{j-1}$ and furnishes us with a homotopy $\tilde u^t\colon K_j\times P\to \R^n$ $(t\in [0,1])$ of nonflat conformal minimal immersions $\tilde u_p^t:=\tilde u^t(\cdot,p)\colon K_j\to\R^n$ satisfying the following conditions.
\begin{enumerate}[\rm (a)]
\item  $\tilde u_p^t=u_p^{0,j-1}|_{K_j}$ for all $(p,t)\in (P\times\{0\})\cup(Q\times[0,1])$.

\smallskip
\item  $|\tilde u_p^t(x)-u_p^{t,j-1}(x)|<\epsilon_j/2$ for all $x\in K_{j-1}$ and $(p,t)\in P\times[0,1]$.

\smallskip
\item  $\Flux(\tilde u_p^t)=F_p^t|_{K_j}$ for all $(p,t)\in P\times[0,1]$. 
\end{enumerate}

Next, choose a compact set $\Delta\subset \mathring K_j$ with $K_{j-1}\subset \mathring \Delta$; so $x_0\in \mathring \Delta$.  Lemma \ref{lem:distance} applies with $K_j$ and $\Delta$ providing a homotopy $u^{t,j}\colon K_j\times P\to\R^n$ $(t\in [0,1])$ of nonflat conformal minimal immersions $u_p^{t,j}:=u^{t,j}(\cdot,p)\colon K_j\to\R^n$ satisfying the following conditions.
\begin{enumerate}[\rm (a)]
\item[\rm (d)]  $u_p^{t,j}=\tilde u_p^t$ for all $(p,t)\in (P\times\{0\})\cup (Q\times[0,1])$.

\smallskip
\item[\rm (e)]  $|u_p^{t,j}(x)-\tilde u_p^t(x)|<\epsilon_j/2$ for all $x\in \Delta\supset K_{j-1}$ and $(p,t)\in P\times[0,1]$.

\smallskip
\item[\rm (f)]  $\Flux(u_p^{t,j})=\Flux(\tilde u_p^t)$ for all $(p,t)\in P\times[0,1]$.

\smallskip
\item[\rm (g)]  $\dist_{u_p^{t,j}}(x_0,bK_j)>j$ for all $(p,t) \in T_j\times[\frac1{j+1},1]$.
\end{enumerate}
Condition {\rm (A$_j$)} is implied by {\rm (a)} and {\rm (d)}; {\rm (B$_j$)} by {\rm (b)} and {\rm (e)}; {\rm (C$_j$)} by {\rm (g)}; and {\rm (E$_j$)} by {\rm (c)} and {\rm (f)}. Recall that {\rm (D$_j$)} is already granted.

This closes the inductive construction and completes the proof of Theorem \ref{th:main}.


\section{Surfaces of finite topological type}
\label{sec:finite}

\noindent
In this section, we prove Corollary \ref{cor:weak-eq-square}(b), assuming that the open Riemann surface $M$ is of finite topological type.  Recall that this means that $M$ has the homotopy type of a bouquet of finitely many circles or, equivalently by Stout's theorem \cite[Theorem 8.1]{Stout1965TAMS}, that $M$ can be obtained from a compact Riemann surface by removing a finite number of mutually disjoint points and closed discs.

A weak homotopy equivalence between spaces that are absolute neighbourhood retracts (ANRs) in the category of metrisable spaces is a genuine homotopy equivalence \cite[Theorem 15]{Palais1966}.  The spaces $\CMInf(M,\R^n)$ and $\RNCnf(M,\C^n)$ are ANRs \cite[Theorem 6.1]{ForstnericLarusson2019CAG}, so the following result settles the corollary.

\begin{theorem}  \label{th:anr}
Let $M$ be an open Riemann surface of finite topological type and $n\geq 3$.  The spaces $\CMInfc(M,\R^n)$ and $\RNCnfc(M,\C^n)$ are absolute neighbourhood retracts.
\end{theorem}

The theorem is an immediate consequence of the fact that $\CMInf(M,\R^n)$ and $\RNCnf(M,\C^n)$ are ANRs, the parametric h-principle from Theorem \ref{th:main}, and the following proposition.

\begin{proposition}  \label{pr:subspace-is-anr}
Let $(X, d)$ be a second-countable metric space and $Y$ be a subspace of $X$.  Suppose that whenever $P$ is a finite polyhedron, $Q$ is a subpolyhedron of $P$, $f:P\to X$ is a continuous map with $f(Q)\subset Y$, and $\epsilon>0$, there is a homotopy $f_t:P\to X$, $t\in [0,1]$, with $f_0=f$, $f_1(P)\subset Y$, and $f_t=f$ on $Q$ and $d(f_t,f)<\epsilon$ on $P$ for all $t\in [0,1]$.  Then, if $X$ is an ANR, so is $Y$.
\end{proposition}

\begin{proof}
We use the Dugundji-Lefschetz characterisation of the ANR property for second-countable metrisable spaces (\cite[Theorem 5.2.1]{vanMill1989}; for more background, see \cite{Larusson2015PAMS}).  Let $\mathscr U$ be an open cover of $Y$.  Take $\mathscr U$ to be the restriction to $Y$ of an open cover $\mathscr U_0$ of $X$.  We need to produce a refinement $\mathscr V$ of $\mathscr U$ such that if $A$ is a simplicial complex, countable and locally finite, with a subcomplex $B$ containing all the vertices of $A$, then every continuous map $\phi_0:B\to Y$ such that for each simplex $\sigma$ of $A$, $\phi_0(\sigma\cap B)$ lies in an element of $\mathscr V$, extends to a continuous map $\phi:A\to Y$ such that for each simplex $\sigma$ of $A$, $\phi(\sigma)$ lies in an element of $\mathscr U$.

Since $X$ is an ANR by assumption, the open cover $\mathscr U_0$ of $X$ has a refinement $\mathscr V_0$ as in the Dugundji-Lefschetz characterisation.  Let $\mathscr V$ be the restriction of $\mathscr V_0$ to $Y$.  Let $A$, $B$, and $\phi_0$ be as above.  We will show that $\phi_0$ extends to a continuous map $\phi:A\to Y$ such that for each simplex $\sigma$ of $A$, $\phi(\sigma)$ lies in an element of $\mathscr U$.  We do know that $\phi_0$ extends to a continuous map $\psi:A\to X$ such that for each simplex $\sigma$ of $A$, $\psi(\sigma)$ lies in an element of $\mathscr U_0$.

It suffices to prove the following.  Let $P_1 \subset P_2 \subset \cdots$ be finite subcomplexes exhausting $A$ with $P_n\subset \mathring P_{n+1}$ for all $n\geq 1$, and let $\epsilon_1, \epsilon_2, \ldots >0$.  Then there is a continuous extension $\phi:A\to Y$ of $\phi_0$ with $d(\phi,\psi)<\epsilon_n$ on $P_n\setminus P_{n-1}$ for all $n\geq 1$ (take $P_0=\varnothing$).  We may assume that $\epsilon_2 > \epsilon_3 >\cdots$ and $\epsilon_1<\tfrac 1 2\epsilon_3$.

For each $n\geq 1$, let $\lambda_n:A\to [0,1]$ be a continuous function with $\lambda_n=1$ on $P_n$ and with support in $P_{n+1}$.  

To start the inductive construction of $\phi$, find a homotopy $f_t:P_2 \to X$, $t\in [0,1]$, with $f_0=\psi$, $f_1(P_2)\subset Y$, and, for all $t\in [0,1]$, $f_t=\phi_0$ on $P_2\cap B$ and $d(f_t,\psi)<\epsilon_1$ on $P_2$.  Define $\phi_1:A\to X$ by $\phi_1(a)=f_{\lambda_1(a)}(a)$ for $a\in P_2$ and $\phi_1=\psi$ on $A\setminus P_2$.  Then $\phi_1$ is a continuous extension of $\phi_0$ with $\phi_1(P_1)\subset Y$ and $d(\phi_1, \psi)<\epsilon_1$ on~$A$.

Next, find a homotopy $f_t:P_3 \to X$, $t\in [0,1]$, with $f_0=\phi_1$, $f_1(P_3)\subset Y$, and, for all $t\in [0,1]$, $f_t=\phi_0$ on $P_3\cap B$, $f_t=\phi_1$ on $P_1$, and $d(f_t,\phi_1)<\tfrac 1 2\epsilon_3$.  Define $\phi_2:A\to X$ by $\phi_2(a)=f_{\lambda_2(a)}(a)$ for $a\in P_3$ and $\phi_2=\psi$ on $A\setminus P_3$.  Then $\phi_2$ is a continuous extension of $\phi_0$ with $\phi_2=\phi_1$ on $P_1$, $\phi_2(P_2)\subset Y$, and $d(\phi_2, \phi_1)<\tfrac 1 2\epsilon_3$ on~$A$.

Continuing in this way, we obtain continuous maps $\phi_n:A\to X$, $n\geq 1$, that extend $\phi_0$, such that $\phi_{n+1}=\phi_n$ on $P_n$, $\phi_n(P_n)\subset Y$, $\phi_n=\psi$ on $A\setminus P_{n+1}$, and, for $n\geq 2$, $d(\phi_n, \phi_{n-1})<\tfrac 1 2\epsilon_{n+1}$ on $A$.  The limit of $\phi_n$ as $n\to\infty$ is a continuous map $\phi:A\to Y$ that extends $\phi_0$.  Also, $d(\phi, \psi)=d(\phi_1, \psi)<\epsilon_1$ on $P_1$, 
\[ d(\phi, \psi) = d(\phi_2, \psi) \leq d(\phi_2, \phi_1) + d(\phi_1, \psi) < \tfrac 1 2 \epsilon_3 + \epsilon_1 < \epsilon_2 \]
on $P_2\setminus P_1$, and for $n\geq 3$,
\begin{eqnarray*}
d(\phi, \psi)  & = & d(\phi_n, \psi) \\
& \leq & d(\phi_n,\phi_{n-1}) + d(\phi_{n-1},\phi_{n-2}) + \cdots + d(\phi_2, \phi_1) + d(\phi_1, \psi) \\ 
& = & d(\phi_n,\phi_{n-1})+d(\phi_{n-1},\phi_{n-2}) 
\; < \; \tfrac 1 2\epsilon_{n+1} + \tfrac 1 2\epsilon_n \; < \; \epsilon_n
\end{eqnarray*}
on $P_n\setminus P_{n-1}$.
\end{proof}


\section{Full immersions}
\label{sec:full}

\noindent
In this final section, we show how to adapt our results to full immersions in place of nonflat immersions.  Recall that a conformal minimal immersion $u:M\to \R^n$ is said to be full if $\psi(u): M\to\boldA_*$ is full, meaning that the $\C$-linear span of $\psi(u)(M)$ is all of $\C^n$.  Likewise, a holomorphic null curve $\Phi:M\to\C^n$ is full if $\phi(\Phi): M\to\boldA_*$ is full. Here, the maps $\phi$ and $\psi$ are those introduced at the end of Section \ref{sec:intro}, just above Corollary \ref{cor:homotopy-type-of-complete}.  We denote by $\CMIf(M,\R^n)$ the subspace of  $\CMInf(M,\R^n)$ consisting of full immersions. The notation for the subspaces of full immersions appearing in the following theorem should be obvious.

\begin{theorem}  \label{th:all-results-work-for-full}
{\rm (a)}  The parametric h-principle in Theorem \ref{th:main} holds for full immersions in place of nonflat ones.

{\rm (b)}  The flux map $\Flux\colon \CMIfc(M,\R^n)\to H^1(M,\R^n)$ is a Serre fibration.

{\rm (c)}  Let $M$ be an open Riemann surface and $n\geq 3$.  The maps in the diagram
\[ \xymatrix{
\RNCfc(M,\C^n)  \ar@{^{(}->}[r] \ar@{^{(}->}[d]  &  \CMIfc(M,\R^n) \ar@{^{(}->}[d] \\ 
\RNCf(M,\C^n)   \ar@{^{(}->}[r]   &  \CMIf(M,\R^n)  \ar[d]_\psi  \\
\NCf(M,\C^n)  \ar[r]^\phi  \ar[u]^\Re  &  \Of(M,\boldA_*)  \ar@{^{(}->}[r]  & \Oscr(M,\boldA_*) \ar@{^{(}->}[r] & \Cscr(M,\boldA_*)\\ 
} \]
are weak homotopy equivalences.

{\rm (d)}  If $M$ is of finite topological type, then the maps are homotopy equivalences.
\end{theorem}

First, we note that if $u_p$ in Theorem \ref{th:main} is full for all $p\in P$, then sufficiently close approximation on a neighbourhood of a suitable finite subset of $M$ using (iii) implies that $u_p^t$ is full for all $(p,t)\in P\times[0,1]$.  Thus Theorem \ref{th:main} holds for full immersions in place of nonflat ones and (b) follows immediately.  The parametric h-principle \cite[Theorem 4.1]{ForstnericLarusson2019CAG} holds for full immersions by the same argument.  It follows that the inclusions in the square
\[ \xymatrix{
\RNCfc(M,\C^n)  \ar@{^{(}->}[r] \ar@{^{(}->}[d]  &  \CMIfc(M,\R^n) \ar@{^{(}->}[d] \\ 
\RNCf(M,\C^n)   \ar@{^{(}->}[r]   &  \CMIf(M,\R^n) } \]
are weak homotopy equivalences.

As noted in \cite{ForstnericLarusson2019CAG} for nonflat immersions, by continuity in the compact-open topology of the Hilbert transform that takes $u\in \RNCf(M, \C^n)$ to its harmonic conjugate $v$ with $v(x) = 0$, where $x \in M$ is any chosen base point, the real part map $\Re:\NCf(M, \C^n)\to \RNCf(M, \C^n)$ is a homotopy equivalence.  To see that the map $\phi:\NCf(M,\C^n)\to \Of(M,\boldA_*)$ is a weak homotopy equivalence, factor it as
\[ \NCf(M,\C^n)\to \{\Phi\in \NCf(M,\C^n):\Phi(p)=0\} \\ \overset\phi\to \Oscr_{\mathrm{full},0}(M, \boldA_*) \hookrightarrow \Of(M,\boldA_*), \]
where $\Oscr_{\mathrm{full},0}(M, \boldA_*)$ denotes the space of full holomorphic maps $M\to \boldA_*$ with vanishing periods, and note that the first map $\Phi\mapsto \Phi-\Phi(p)$ is a homotopy equivalence, the second a homeomorphism, and the third a weak homotopy equivalence by the parametric h-principle \cite[Theorem 5.3]{ForstnericLarusson2019CAG} adapted to full maps in place of nonflat maps in the way described above.  To complete the proof of Theorem \ref{th:all-results-work-for-full}(c), the general position theorem \cite[Theorem 5.4]{ForstnericLarusson2019CAG} is easily adapted to full maps so as to imply that the inclusion $\Of(M,\boldA_*) \hookrightarrow \Oscr(M, \boldA_*)$ is a weak homotopy equivalence.  In fact, the proof of \cite[Theorem 5.4]{ForstnericLarusson2019CAG} yields the following stronger general position theorem.

\begin{theorem}  \label{th:general-position}
Let $M$ be an open Riemann surface, $K\subset M$ be compact, $P$ be a compact metric space, $Q$ be a closed subspace of $P$, $f: P\to \mathscr O(M,\boldA_*)$ be a continuous map, and $\epsilon>0$.  There is a homotopy $f^t : P\to \mathscr O(M,\boldA_*)$, $t\in [0,1]$, such that:
\begin{enumerate}[\rm (1)]
\item  $f_p^t =f_p$ for all $(p,t)\in (P\times \{0\}) \cup (Q\times [0,1])$.
\item  $f_p^t\in \mathscr O(M, \boldA_*)$ is full for all $(p,t) \in (P\setminus Q)\times (0,1]$.
\item  $\lvert f_p^t(x) - f_p(x) \rvert <\epsilon$ for all $x\in K$ and $(p,t) \in P\times [0,1]$.
\end{enumerate}
\end{theorem}

Finally, we assume that $M$ is of finite topological type.  The fact that an open subspace of an ANR is an ANR implies that the spaces $\RNCf(M,\C^n)$, $\CMIf(M,\R^n)$, $\NCf(M,\C^n)$, and $\Of(M,\boldA_*)$ are ANRs.  Arguing as in Section \ref{sec:finite}, we conclude that $\RNCfc(M,\C^n)$ and $\CMIfc(M,\R^n)$ are also ANRs.  This completes the proof of Theorem \ref{th:all-results-work-for-full}.


\subsection*{Acknowledgements}
The authors gratefully acknowledge past joint work with Franc Forstneri\v c on which this paper is based.

\noindent
A.\ Alarc\'on was partially supported by the State Research Agency (AEI) and European Regional Development Fund (FEDER) via the grant no.\ MTM2017-89677-P, PID2020-117868GB-I00, and PID2023-150727NB-I00, and the ``Maria de Maeztu'' Excellence Unit IMAG, reference CEX2020-001105-M, funded by MCIN/AEI/10.13039/501100011033/;  the Junta de Andaluc\'ia grant no. P18-FR-4049; and the Junta de Andaluc\'ia - FEDER grant no. A-FQM-139-UGR18; Spain. 



\begin{thebibliography}{10}

\bibitem{AlarconCastro-Infantes2019APDE}
A.~Alarc{\'o}n and I.~Castro-Infantes.
\newblock Interpolation by conformal minimal surfaces and directed holomorphic
  curves.
\newblock {\em Anal. PDE}, 12(2):561--604, 2019.

\bibitem{AlarconFernandezLopez2012CMH}
A.~Alarc{\'o}n, I.~Fern{\'a}ndez, and F.~J. L{\'o}pez.
\newblock Complete minimal surfaces and harmonic functions.
\newblock {\em Comment. Math. Helv.}, 87(4):891--904, 2012.

\bibitem{AlarconFernandezLopez2013CVPDE}
A.~Alarc{\'o}n, I.~Fern{\'a}ndez, and F.~J. L{\'o}pez.
\newblock Harmonic mappings and conformal minimal immersions of {R}iemann
  surfaces into {$\mathbb{R}^{\rm N}$}.
\newblock {\em Calc. Var. Partial Differential Equations}, 47(1-2):227--242,
  2013.

\bibitem{AlarconForstneric2014IM}
A.~Alarc{\'o}n and F.~Forstneri\v{c}.
\newblock Null curves and directed immersions of open {R}iemann surfaces.
\newblock {\em Invent. Math.}, 196(3):733--771, 2014.

\bibitem{AlarconForstneric2018Crelle}
A.~Alarc{\'o}n and F.~Forstneri\v{c}.
\newblock Every conformal minimal surface in {$\mathbb R^3$} is isotopic to the
  real part of a holomorphic null curve.
\newblock {\em J. reine angew. Math.}, 740:77--109, 2018.

\bibitem{AlarconForstneric2019JAMS}
A.~Alarc{\'o}n and F.~Forstneri\v{c}.
\newblock New complex analytic methods in the theory of minimal surfaces: a
  survey.
\newblock {\em J. Aust. Math. Soc.}, 106(3):287--341, 2019.

\bibitem{AlarconForstneric2019RMI}
A.~Alarc\'{o}n and F.~Forstneri\v{c}.
\newblock The {C}alabi-{Y}au problem for {R}iemann surfaces with finite genus
  and countably many ends.
\newblock {\em Rev. Mat. Iberoam.}, 37(4):1399--1412, 2021.

\bibitem{AlarconForstnericLopez2016MZ}
A.~Alarc{\'o}n, F.~Forstneri\v{c}, and F.~J. L{\'o}pez.
\newblock Embedded minimal surfaces in {$\mathbb {R}^n$}.
\newblock {\em Math. Z.}, 283(1-2):1--24, 2016.

\bibitem{AlarconForstnericLopez2019JGEA}
A.~Alarc{\'o}n, F.~Forstneri\v{c}, and F.~J. L{\'o}pez.
\newblock Every meromorphic function is the {G}auss map of a conformal minimal
  surface.
\newblock {\em J. Geom. Anal.}, 29(4):3011--3038, 2019.

\bibitem{AlarconForstnericLopez2021}
A.~Alarc\'{o}n, F.~Forstneri\v{c}, and F.~J. L\'{o}pez.
\newblock {\em Minimal surfaces from a complex analytic viewpoint}.
\newblock Springer Monographs in Mathematics. Springer, Cham, 2021.

\bibitem{AlarconLarusson2017IJM}
A.~Alarc{\'o}n and F.~L{\'a}russon.
\newblock Representing de {R}ham cohomology classes on an open {R}iemann
  surface by holomorphic forms.
\newblock {\em Internat. J. Math.}, 28(9):1740004, 12, 2017.

\bibitem{AlarconLarusson2022}
A.~Alarc{\'o}n and F.~L{\'a}russon.
\newblock The space of {G}auss maps of complete minimal surfaces.
\newblock {\em Ann. Sc. Norm. Super. Pisa Cl. Sci. (5)}, 25(2):669--688, 2024.

\bibitem{AlarconLopez2012JDG}
A.~Alarc{\'o}n and F.~J. L{\'o}pez.
\newblock Minimal surfaces in {$\mathbb R^3$} properly projecting into
  {$\mathbb R^2$}.
\newblock {\em J. Differential Geom.}, 90(3):351--381, 2012.

\bibitem{AlarconLopez2013JGEA}
A.~Alarc{\'o}n and F.~J. L{\'o}pez.
\newblock Proper holomorphic embeddings of {R}iemann surfaces with arbitrary
  topology into {${\mathbb C}^2$}.
\newblock {\em J. Geom. Anal.}, 23(4):1794--1805, 2013.

\bibitem{AlarconLopez2021APDE}
A.~Alarc\'{o}n and F.~J. L\'{o}pez.
\newblock Algebraic approximation and the {M}ittag-{L}effler theorem for
  minimal surfaces.
\newblock {\em Anal. PDE}, 15(3):859--890, 2022.

\bibitem{Bishop1958PJM}
E.~Bishop.
\newblock Subalgebras of functions on a {R}iemann surface.
\newblock {\em Pacific J. Math.}, 8:29--50, 1958.

\bibitem{Castro-Infantes2021}
I.~Castro-Infantes.
\newblock Interpolation by complete minimal surfaces whose {G}auss map misses
  two points.
\newblock {\em J. Geom. Anal.}, 31(5):5395--5417, 2021.

\bibitem{Castro-InfantesChenoweth2020}
I.~Castro-Infantes and B.~Chenoweth.
\newblock Carleman approximation by conformal minimal immersions and directed
  holomorphic curves.
\newblock {\em J. Math. Anal. Appl.}, 484(2):123756, 2020.

\bibitem{EliashbergMishachev2002}
Y.~Eliashberg and N.~Mishachev.
\newblock {\em Introduction to the {$h$}-principle}, volume~48 of {\em Graduate
  Studies in Mathematics}.
\newblock Amer. Math. Soc., Providence, RI, 2002.

\bibitem{Forstneric2003FM}
F.~Forstneri\v{c}.
\newblock The {O}ka principle for multivalued sections of ramified mappings.
\newblock {\em Forum Math.}, 15(2):309--328, 2003.

\bibitem{Forstneric2017E}
F.~Forstneri\v{c}.
\newblock {\em Stein manifolds and holomorphic mappings (The homotopy principle
  in complex analysis)}, volume~56 of {\em Ergebnisse der Mathematik und ihrer
  Grenzgebiete. 3. Folge}.
\newblock Springer, Cham, second edition, 2017.

\bibitem{ForstnericLarusson2019CAG}
F.~Forstneri\v{c} and F.~L{\'a}russon.
\newblock {The parametric \(h\)-principle for minimal surfaces in
  \(\mathbb{R}^n\) and null curves in \(\mathbb{C}^n\).}
\newblock {\em {Commun. Anal. Geom.}}, 27(1):1--45, 2019.

\bibitem{Gromov1973}
M.~Gromov.
\newblock Convex integration of differential relations. {I}.
\newblock {\em Izv. Akad. Nauk SSSR Ser. Mat.}, 37:329--343, 1973.

\bibitem{Gromov1986}
M.~Gromov.
\newblock {\em Partial differential relations}, volume~9 of {\em Ergebnisse der
  Mathematik und ihrer Grenzgebiete, 3. Folge}.
\newblock Springer-Verlag, Berlin, 1986.

\bibitem{JorgeXavier1980AM}
L.~P. d.~M. Jorge and F.~Xavier.
\newblock A complete minimal surface in {${\bf R}^{3}$} between two parallel
  planes.
\newblock {\em Ann. of Math. (2)}, 112(1):203--206, 1980.

\bibitem{Larusson2015PAMS}
F.~L{\'a}russon.
\newblock Absolute neighbourhood retracts and spaces of holomorphic maps from
  {S}tein manifolds to {O}ka manifolds.
\newblock {\em Proc. Amer. Math. Soc.}, 143(3):1159--1167, 2015.

\bibitem{LopezRos1991JDG}
F.~J. L{\'o}pez and A.~Ros.
\newblock On embedded complete minimal surfaces of genus zero.
\newblock {\em J. Differential Geom.}, 33(1):293--300, 1991.

\bibitem{Osserman1986}
R.~Osserman.
\newblock {\em A survey of minimal surfaces}.
\newblock Dover Publications, Inc., New York, second edition, 1986.

\bibitem{Palais1966}
R.~S. Palais.
\newblock Homotopy theory of infinite dimensional manifolds.
\newblock {\em Topology}, 5:1--16, 1966.

\bibitem{Stout1965TAMS}
E.~L. Stout.
\newblock Bounded holomorphic functions on finite {R}iemann surfaces.
\newblock {\em Trans. Amer. Math. Soc.}, 120:255--285, 1965.

\bibitem{vanMill1989}
J.~van Mill.
\newblock {\em Infinite-dimensional topology. Prerequisites and introduction},
  volume~43 of {\em North-Holland Mathematical Library}.
\newblock North-Holland Publishing Co., Amsterdam, 1989.

\end{thebibliography}


\medskip
\noindent Antonio Alarc\'{o}n

\noindent Departamento de Geometr\'{\i}a y Topolog\'{\i}a e Instituto de Matem\'aticas (IMAG), Universidad de Granada, Campus de Fuentenueva s/n, E--18071 Granada, Spain

\noindent  e-mail: {\tt alarcon@ugr.es}

\bigskip
\noindent Finnur L\'arusson

\noindent School of Mathematical Sciences, University of Adelaide, Adelaide SA 5005, Australia

\noindent  e-mail: {\tt finnur.larusson@adelaide.edu.au}

\end{document}